\documentclass[11pt,a4paper]{article}
\usepackage[left=2cm,right=2cm,top=2cm,bottom=2cm]{geometry}
\usepackage{graphicx}
\usepackage{multirow}
\usepackage{amsmath,amssymb,amsfonts}
\usepackage{amsthm}
\usepackage{mathrsfs}
\usepackage[title]{appendix}
\usepackage{xcolor}
\usepackage{textcomp}
\usepackage{manyfoot}
\usepackage{booktabs}
\usepackage{algorithm}
\usepackage{algorithmicx}
\usepackage{algpseudocode}
\usepackage{listings}
\usepackage{bm}
\usepackage{hyperref}
\ifdefined\debug
  \usepackage[inline]{showlabels}
  \showlabels{ref}
  \showlabels{cite}
  \usepackage{fancyhdr}
\fi

\theoremstyle{plain}
\newtheorem{theorem}{Theorem}[section]
\newtheorem{proposition}[theorem]{Proposition}%
\newtheorem{lemma}[theorem]{Lemma}
\newtheorem{corollary}[theorem]{Corollary}

\theoremstyle{definition}
\newtheorem{definition}[theorem]{Definition}
\newtheorem{remark}[theorem]{Remark}

\newtheorem{convention}[theorem]{Convention}
\theoremstyle{plain}
\newtheorem{maintheorem}{Theorem}
\ifdefined\debug
  \newcommand{\debugtext}[1]{\textcolor{red}{#1}}
\else
  \newcommand{\debugtext}[1]{\ignorespaces}
\fi

\counterwithin{figure}{section}
\counterwithin{table}{section}

\numberwithin{equation}{section}

\begin{document}

\title{\textbf{
Infinitely Many Attracting Periodic Circles in Higher Dimensions
}}

\author{Shuntaro Tomizawa
\footnote{
tomizawa.shuntaro@gmail.com, Graduate School of Mathematical Sciences, The University of Tokyo, 3-8-1 Komaba, Meguro, Tokyo, 153-8914, Japan
}}

\maketitle

\noindent \textbf{Abstract.}
We study $C^r$ ($5 \le r \le \infty$) diffeomorphisms on closed manifolds of dimension at least three with a heteroclinic cycle between two hyperbolic periodic points. At each point, the unstable direction is one dimensional, and the stable and unstable eigenvalues closest to $1$ in modulus are real and simple. One heteroclinic connection is transverse and the other is non-transverse, and the product of those two eigenvalues is less than $1$ at one point and greater than $1$ at the other. Arbitrarily close to such a map, there are open sets in which a residual subset of diffeomorphisms has infinitely many attracting normally hyperbolic periodic circles. The proof uses a rescaling to the standard H\'enon map and a corrected formula for the Lyapunov coefficient on its Neimark-Sacker (Andronov-Hopf) line.

\medskip


\medskip


\tableofcontents

%%%%%%%%%%%%%%%%%%%%%%%%%%% body start %%%%%%%%%%%%%%%%%%%%%%%%%%%%%%

\section{Introduction}
\label{sec-intro}

Homoclinic and heteroclinic tangencies are standard sources of persistent nonhyperbolic dynamics. They generate both Newhouse phenomena and invariant circle bifurcations near tangencies \cite{Newhouse1974,GG2000,GG2004,GSS2002,GSS2006,Tatjer2001,GST2008}. In this paper we study a higher-dimensional heteroclinic cycle in Figure~\ref{fig-DEcycle}. Our main result shows that every diffeomorphism with such a cycle belongs to the closure of an open set in which diffeomorphisms with infinitely many attracting normally hyperbolic periodic circles form a residual subset. The proof combines a two-parameter unfolding, a Newhouse domain argument, and a rescaling to the standard H\'enon map based on the corrected Lyapunov coefficient.

Let $r \in \mathbb{Z}_{>0} \sqcup \{\infty\}$, and let $f$ be a $C^r$ diffeomorphism of a closed $C^r$ manifold $M_\mathrm{ph}$ with a Riemannian metric. For $n \in \mathbb{Z}$, set
\[
\mathbb{Z}_{> n} := \{n + 1, n + 2, \dots\},
\qquad
\mathbb{Z}_{\ge n} := \{n, n + 1, \dots\}.
\]
For a hyperbolic periodic point $O$, define
\begin{align*}
W^s(O) &:= \{P \in M_\mathrm{ph} \mid \mathrm{dist}(f^k(P), f^k(O)) \xrightarrow{k \to \infty} 0\}, \\
W^u(O) &:= \{P \in M_\mathrm{ph} \mid \mathrm{dist}(f^{-k}(P), f^{-k}(O)) \xrightarrow{k \to \infty} 0\}, \\
W^s(\mathrm{Orb}(O)) &:= \bigsqcup_{k = 0}^{\mathrm{per}(O) - 1} W^s(f^k(O)), \\
W^u(\mathrm{Orb}(O)) &:= \bigsqcup_{k = 0}^{\mathrm{per}(O) - 1} W^u(f^k(O)).
\end{align*}
Here $\mathrm{dist}$ is the Riemannian distance, $\mathrm{Orb}(P) := \{f^k(P) \mid k \in \mathbb{Z}\}$ is the orbit of a point $P$, and $\mathrm{per}(P)$ is the period of a periodic point $P$. The $u$-index of $O$ is $\dim W^u(O)$.

\begin{definition}
Let $O_1^*$ and $O_2^*$ be distinct hyperbolic periodic points of $f$ with the same $u$-index. Assume that
\[
W^u(\mathrm{Orb}(O_1^*)) \cap W^s(\mathrm{Orb}(O_2^*)) \neq \emptyset,
\qquad
W^u(\mathrm{Orb}(O_2^*)) \cap W^s(\mathrm{Orb}(O_1^*)) \neq \emptyset.
\]
Choose
\[
M_{1 \to 2}^* \in W^u(\mathrm{Orb}(O_1^*)) \cap W^s(\mathrm{Orb}(O_2^*)),
\qquad
M_{2 \to 1}^* \in W^u(\mathrm{Orb}(O_2^*)) \cap W^s(\mathrm{Orb}(O_1^*)),
\]
and set
\[
\Gamma_{1 \to 2}^* := \mathrm{Orb}(M_{1 \to 2}^*),
\qquad
\Gamma_{2 \to 1}^* := \mathrm{Orb}(M_{2 \to 1}^*).
\]
Then
\[
\Gamma^* := \Gamma_{1 \to 2}^* \sqcup \Gamma_{2 \to 1}^* \sqcup \mathrm{Orb}(O_1^*) \sqcup \mathrm{Orb}(O_2^*)
\]
is the \emph{heteroclinic cycle} associated with $\mathrm{Orb}(O_1^*)$ and $\mathrm{Orb}(O_2^*)$.
\end{definition}

\begin{definition}
Let $O$ be a hyperbolic periodic point. Its \emph{multipliers} are the eigenvalues of $Df^{\mathrm{per}(O)}(O)$. The \emph{stable leading multipliers} are the stable multipliers whose moduli are closest to $1$. The \emph{unstable leading multipliers} are defined similarly. Together they are the \emph{leading multipliers} of $O$.
\end{definition}

\begin{definition}
Assume that $\Gamma^*$ is the heteroclinic cycle associated with $\mathrm{Orb}(O_1^*)$ and $\mathrm{Orb}(O_2^*)$.
We say that $\Gamma^*$ is of \emph{type two bi-saddles} if, for each $\ell \in \{1,2\}$, the point $O_\ell^*$ has a real simple stable leading multiplier $\lambda_\ell^*$ and a real simple unstable leading multiplier $\gamma_\ell^*$. In this case set
\[
\sigma_\ell^* := |\lambda_\ell^* \gamma_\ell^*|
\qquad
(\ell \in \{1,2\}).
\]
We say that $\Gamma^*$ is \emph{centrally dissipative-expanding} if it is of type two bi-saddles and
\[
\sigma_1^* < 1 < \sigma_2^*
\qquad
\text{or}
\qquad
\sigma_2^* < 1 < \sigma_1^*.
\]
We say that $\Gamma^*$ has \emph{one-dimensional unstable directions} if it is of type two bi-saddles and the $u$-index of each $O_\ell^*$ is $1$.
\end{definition}

\begin{definition}
For $\imath,\jmath \in \{1,2\}$ with $\imath \neq \jmath$, a \emph{heteroclinic connection} from $\mathrm{Orb}(O_\imath^*)$ to $\mathrm{Orb}(O_\jmath^*)$ is an orbit
\[
\Gamma = \mathrm{Orb}(M^*), \qquad M^* \in W^u(\mathrm{Orb}(O_\imath^*)) \cap W^s(\mathrm{Orb}(O_\jmath^*)).
\]
It is \emph{transverse} if for some, equivalently every, $M \in \Gamma$,
\[
T_M W^u(\mathrm{Orb}(O_\imath^*)) + T_M W^s(\mathrm{Orb}(O_\jmath^*)) = T_M M_\mathrm{ph}.
\]
Otherwise it is \emph{non-transverse}. We say that $\Gamma^*$ is \emph{transversal and non-transversal} if one of its two heteroclinic connections is transverse and the other is non-transverse.
\end{definition}

\begin{convention}
In this case, we relabel if necessary so that the connection from $\mathrm{Orb}(O_1^*)$ to $\mathrm{Orb}(O_2^*)$ is transverse at $M_{1 \to 2}^*$ and the connection from $\mathrm{Orb}(O_2^*)$ to $\mathrm{Orb}(O_1^*)$ is non-transverse at $M_{2 \to 1}^*$. We use this convention throughout.
\end{convention}

In this paper we study centrally dissipative-expanding transversal and non-transversal heteroclinic cycles of type two bi-saddles with one-dimensional unstable directions; see Figure~\ref{fig-DEcycle}.

\begin{figure}[h]
\centering
\includegraphics{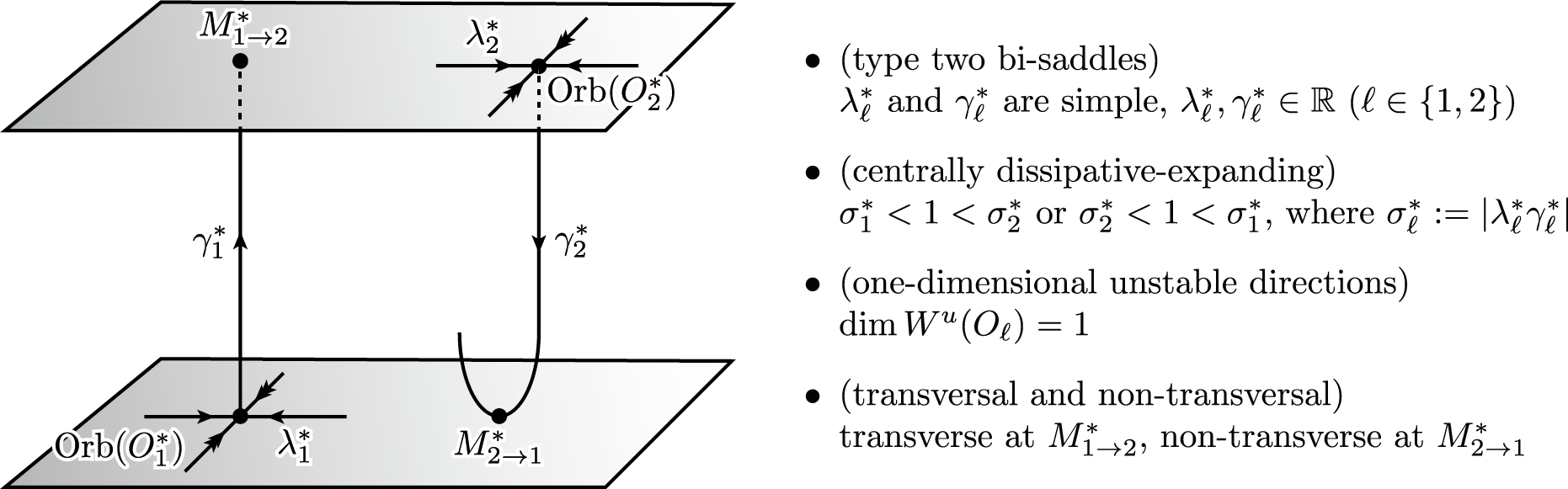}
\caption{The heteroclinic cycle considered in this paper.}
\label{fig-DEcycle}
\end{figure}
% \begin{itemize}
% \item (type two bi-saddles) \\ $\lambda_\ell^*$ and $\gamma_\ell^*$ are simple, $\lambda_\ell^*, \gamma_\ell^* \in \mathbb{R}$ ($\ell \in \{1, 2\}$)
% \item (centrally dissipative-expanding) \\ $\sigma_1^* < 1 < \sigma_2^*$ or $\sigma_2^* < 1 < \sigma_1^*$, where $\sigma_\ell^* := |\lambda_\ell^* \gamma_\ell^*|$
% \item (one-dimensional unstable directions) \\ $\dim W^u(O_\ell) = 1$
% \item (transversal and non-transversal) \\ transverse at $M_{1 \to 2}^*$, non-transverse at $M_{2 \to 1}^*$
% \end{itemize}

We write $\mathrm{Diff}^r(M_\mathrm{ph})$ for the space of $C^r$ diffeomorphisms of $M_\mathrm{ph}$ endowed with the $C^r$ topology.

\begin{definition}
Let $s \in \mathbb{Z}_{>0} \sqcup \{\infty\}$ satisfy $s \le r$, and let $g \in \mathrm{Diff}^r(M_\mathrm{ph})$. A \emph{$C^{s}$ circle} in $M_\mathrm{ph}$ is the image of a $C^{s}$ embedding $\mathbb{S}^1 \longrightarrow M_\mathrm{ph}$, where $\mathbb{S}^1 := \mathbb{R}/\mathbb{Z}$. A $C^{s}$ circle $C$ is a \emph{periodic circle of period $\tau \in \mathbb{Z}_{> 0}$} for $g$ if
\[
g^\tau(C) = C,
\qquad
g^k(C) \cap C = \emptyset
\quad
\text{for every } k \in \{1,2,\dots,\tau - 1\}.
\]
We call $C$ \emph{attracting} if there exists a neighborhood $V$ of $C$ such that
\[
\mathrm{dist}(g^{k\tau}(P), C) \xrightarrow{k \to \infty} 0
\qquad
\text{for every } P \in V.
\]
We call $C$ \emph{normally hyperbolic} if it is immediately relatively $1$-normally hyperbolic for $g^\tau$ in the sense of \cite[Section~1, Definition~1]{HPS1977}.
\end{definition}

\begin{maintheorem}
\label{thm-main}
Let $r \in \mathbb{Z}_{\ge 5} \sqcup \{\infty\}$, let $M_\mathrm{ph}$ be a closed $C^r$ manifold with $\dim(M_\mathrm{ph}) \ge 3$, and let $f \in \mathrm{Diff}^r(M_\mathrm{ph})$ have a centrally dissipative-expanding transversal and non-transversal heteroclinic cycle of type two bi-saddles with one-dimensional unstable directions. Then there exists an open set $\mathcal{U} \subset \mathrm{Diff}^r(M_\mathrm{ph})$ with $f \in \overline{\mathcal{U}}$ such that:
\begin{itemize}
\item if $r < \infty$, then there exists a residual set $\mathcal{R} \subset \mathcal{U}$ such that, for every $g \in \mathcal{R}$ and every $k \in \mathbb{Z}_{>0}$, the diffeomorphism $g$ has an attracting normally hyperbolic periodic $C^r$ circle of period at least $k$;
\item if $r = \infty$, then for every sequence $\{s_k\}_{k=1}^\infty \subset \mathbb{Z}_{\ge 5}$ with $s_k \xrightarrow{k \to \infty} \infty$, there exists a residual set $\mathcal{R}_{\{s_k\}} \subset \mathcal{U}$ such that, for every $g \in \mathcal{R}_{\{s_k\}}$ and every $k \in \mathbb{Z}_{>0}$, the diffeomorphism $g$ has an attracting normally hyperbolic periodic $C^{s_k}$ circle of period at least $k$.
\end{itemize}
\end{maintheorem}

\begin{remark}\leavevmode
\label{rem-LCinc}
\begin{enumerate}
\item[(1)] The periodic circles in Theorem~\ref{thm-main} are created by a Neimark-Sacker (Andronov-Hopf) bifurcation. This is why we require $r \ge 5$.
\item[(2)] If $2 \le r < 5$ and $M_\mathrm{ph}$ is smooth, then the same conclusion as in Theorem~\ref{thm-main} follows by approximating the original diffeomorphism by smooth diffeomorphisms and using a Newhouse domain; see Section~\ref{sec-PMT}.
\item[(3)] Even when $r = \infty$, we state only finite regularity for the periodic circles. The reason is that the argument passes through a local center manifold at a Neimark-Sacker (Andronov-Hopf) point, and such a center manifold is in general only finitely smooth; see \cite[Section~5.10.2]{Robinson1999}.
\item[(4)] When $\dim(M_\mathrm{ph}) = 2$, related invariant circle mechanisms were studied in \cite{GSS2002,GSS2006}. Theorem~\ref{thm-main} extends that picture to higher dimensions. We also correct the Lyapunov coefficient formula used in \cite[Equation~(30)]{GSS2002}; the existence results in those papers remain valid. See Section~\ref{sec-RLC}.
\end{enumerate}
\end{remark}

\subsection{Previous works and our results}

The background of Theorem~\ref{thm-main} has two complementary components. The first is the Newhouse theory of persistent tangencies and infinitely many sinks. In dimension two, Newhouse proved that persistent homoclinic tangencies generate open sets in which diffeomorphisms with infinitely many sinks are dense \cite{Newhouse1974}. Higher dimensional extensions and refinements were obtained by Palis and Viana \cite{PalisViana1994} and by Romero \cite{Romero1995}.

The second component concerns quasi periodic attractors born from tangencies. In dimension two, precursor bifurcation scenarios for homoclinic tangencies to neutral saddles were studied in \cite{G2002}. The birth of invariant circles and the corresponding Lyapunov coefficient computations were analyzed in \cite{GG2000,GG2004}. These ideas were developed in \cite{GSS2002,GSS2006}, yielding Newhouse domains containing diffeomorphisms with infinitely many stable and unstable invariant tori. In dimension three, related precursor configurations were studied in \cite{GO2005,Tatjer2001}, while \cite[Theorem~6]{GST2008} gave multidimensional criteria for the existence of infinitely many stable invariant tori in suitable Newhouse domains.

The present paper isolates a more concrete higher dimensional mechanism: a centrally dissipative-expanding transversal and non-transversal heteroclinic cycle of type two bi-saddles with one-dimensional unstable directions. The intermediate result, stated later as Theorem~\ref{thm-circ}, shows that suitable two-parameter unfoldings of such an orientable cycle contain parameter values with attracting periodic circles of arbitrarily large period. Theorem~\ref{thm-main} is then obtained by combining this local statement with the genericity of \textup{(P1)}--\textup{(P3)} (see Section~\ref{sec-GC}) and a Newhouse domain argument. A second point is that, after correcting the Lyapunov coefficient, the rescaled return map can be taken to be the standard H\'enon map rather than a larger quadratic H\'enon type family; see the next section. This makes the bifurcation analysis and the rescaling argument more direct.

\subsection{Correction of the Lyapunov coefficient}
\label{sec-RLC}

For the precise discussion, see Section~\ref{sec-BFPHM}.

Near a nonresonant Neimark-Sacker (Andronov-Hopf) point, the Lyapunov coefficient must be read from the fully normalized map
\[
\mathrm{tr} \circ \mathrm{map} \circ \mathrm{tr}^{-1} : w \longmapsto \widetilde{w},
\]
not from the partially transformed expression
\[
\mathrm{tr} \circ \mathrm{map} : z \longmapsto \widetilde{w}.
\]
Proposition~\ref{prop-LCfmla} gives the formula used later. As explained in Remark~\ref{rem-LCinc}~(4), omitting $\mathrm{tr}^{-1}$ is the normalization error behind \cite[Equation~(30)]{GSS2002} and the corresponding formula in \cite[the equation before Equation~(50)]{GG2000}. The same omission also appears in \cite[Section~4, Equation~(4.26)]{K2023}, \cite[Chapter~III, Section~1, Exercise~1]{Iooss1979}, and in the quoted formula of \cite[Theorem~4]{MoraRuiz2011}.

For the standard H\'enon map $F_{(M_1,M_2)}: (X, Y) \longmapsto (\overline{X}, \overline{Y})$,
\[
\overline{X} = Y,
\qquad
\overline{Y} = M_1 - M_2 X - Y^2,
\]
the Neimark-Sacker (Andronov-Hopf) line is
\[
L^\omega := \{(M_1,M_2) \in \mathbb{R}^2 \mid M_2 = 1,\ -1 < M_1 < 3\}.
\]
The resonant points on $L^\omega$ are
\[
C_1^\omega = (0,1),
\qquad
C_2^\omega = \left(\frac{5}{4},1\right).
\]
Proposition~\ref{prop-LC4Hen} shows that the corrected coefficient does not vanish on $L^\omega \setminus \{C_1^\omega,C_2^\omega\}$. Hence the standard H\'enon map already yields the nondegenerate Neimark-Sacker (Andronov-Hopf) bifurcations used in this paper. Figure~\ref{fig-bifdiam} shows the graph of the coefficient and the bifurcation diagram.
\begin{figure}[h]
\centering
\begin{minipage}[c]{0.40\linewidth}
\centering
\includegraphics[width=\linewidth]{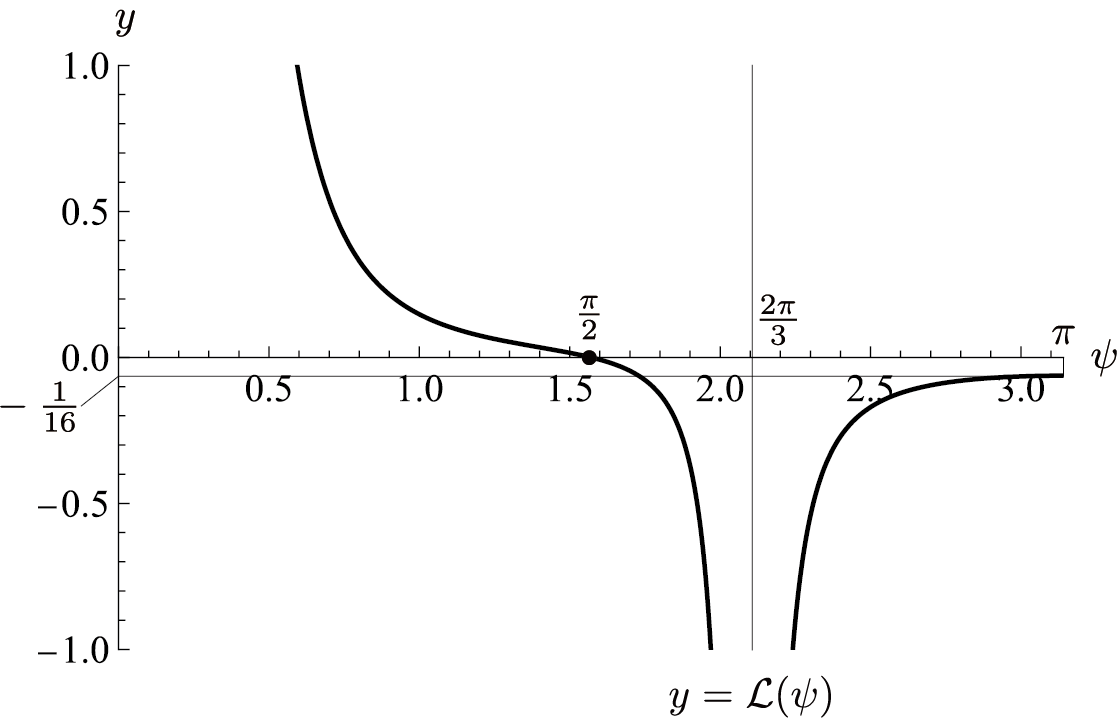}

\small (a)
\end{minipage}
\hfill
\begin{minipage}[c]{0.56\linewidth}
\centering
\includegraphics[width=\linewidth]{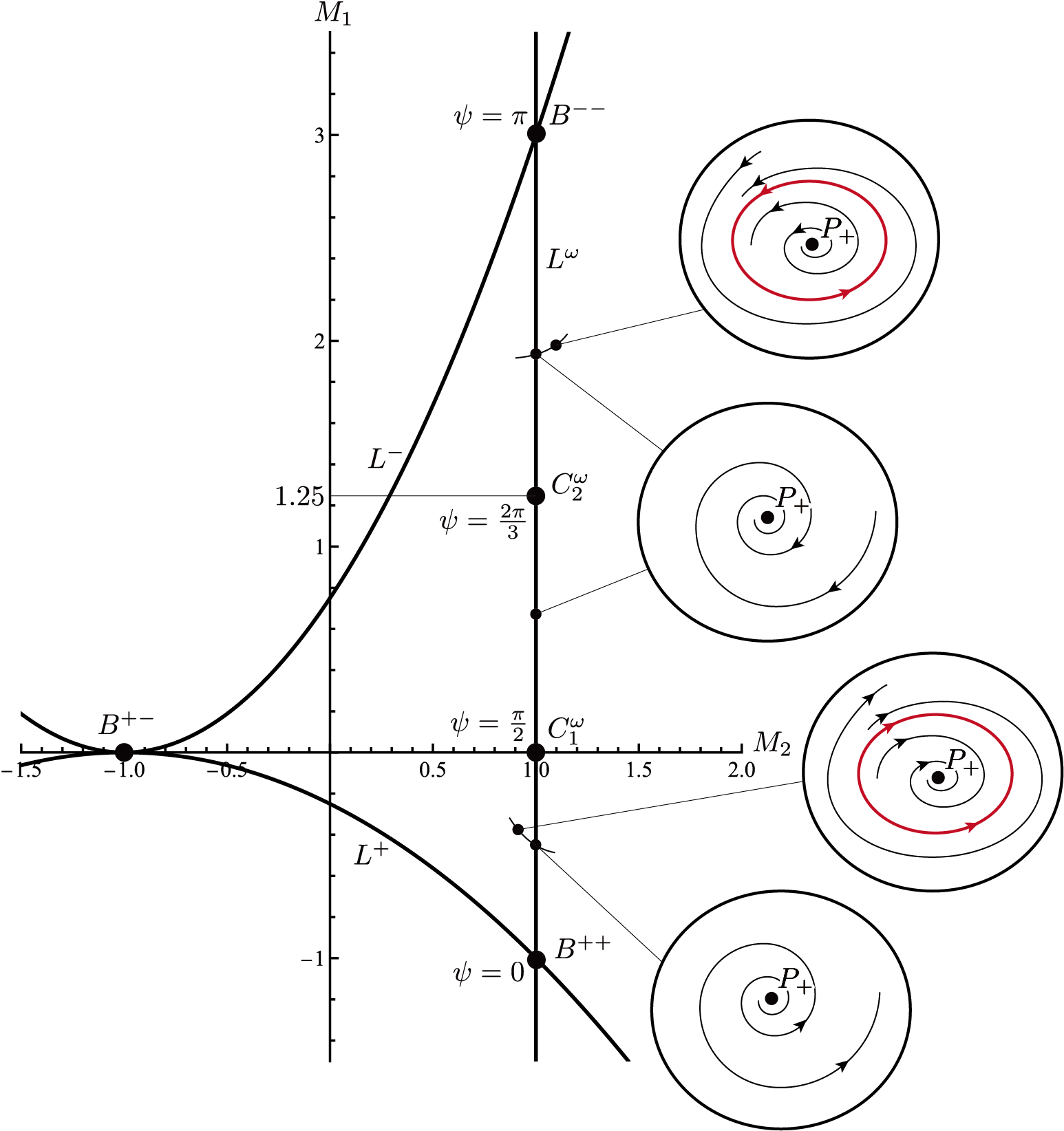}

\small (b)
\end{minipage}
\caption{(a) Graph of the Lyapunov coefficient $\mathcal{L}(\psi)$ for the standard H\'enon map. (b) Bifurcation diagram of the standard H\'enon map. On $L^\omega$, $M_1=\cos^2\psi-2\cos\psi$.}
\label{fig-bifdiam}
\end{figure}

\begin{remark}\leavevmode
\begin{enumerate}
\item[(1)] For our proof, the generalized H\'enon map is unnecessary; compare \cite{GSS2002,GKM2005,GOT2012}. The standard H\'enon map already suffices after the correction.
\item[(2)] This correction does not affect the mechanism in \cite{GSS2002}. The invariant circles used there arise near $B^{--} = (3,1)$, where $P_+$ has the double multiplier $(-1,-1)$, not from a Neimark-Sacker (Andronov-Hopf) bifurcation on $L^\omega$.
\end{enumerate}
\end{remark}

\subsection{Organization of this paper}

Section~\ref{sec-prelim} introduces the geometric setting, the unfoldings, and the return maps. Section~\ref{sec-rescaling} proves the rescaling lemma. Section~\ref{sec-BFPHM} studies fixed point bifurcations of the standard H\'enon map. Section~\ref{sec-PMT} proves Theorem~\ref{thm-main}.
\section{Preliminaries}
\label{sec-prelim}

\subsection{Geometrical settings}
\label{sec-GS}

\subsubsection{Unparametrized local maps and global maps}

Let $f$ satisfy the assumptions of Theorem~\ref{thm-main}. Choose pairwise disjoint small connected open neighborhoods $U_1^*$ and $U_2^*$ of $O_1^*$ and $O_2^*$, respectively. For each $\ell \in \{1,2\}$, define the local return map
\[
T_\ell^* = T_\ell^*(f,U_\ell^*) := f^{\mathrm{per}(O_\ell^*)}|_{U_\ell^*}.
\]
Choose \emph{base points}
\[
M_{1,\mathrm{in}}^*,\ M_{1,\mathrm{out}}^*,\ M_{2,\mathrm{in}}^*,\ M_{2,\mathrm{out}}^*
\]
so that
\begin{align*}
&M_{1,\mathrm{in}}^* \in U_1^* \cap \Gamma_{2 \to 1}^*, \qquad
M_{1,\mathrm{out}}^* \in U_1^* \cap \Gamma_{1 \to 2}^*, \\
&M_{2,\mathrm{in}}^* \in U_2^* \cap \Gamma_{1 \to 2}^*, \qquad
M_{2,\mathrm{out}}^* \in U_2^* \cap \Gamma_{2 \to 1}^*.
\end{align*}

Let
\[
N_{1 \to 2},\ N_{2 \to 1} \in \mathbb{Z}_{>0}
\]
be such that
\[
f^{N_{1 \to 2}}(M_{1,\mathrm{out}}^*) = M_{2,\mathrm{in}}^*,
\qquad
f^{N_{2 \to 1}}(M_{2,\mathrm{out}}^*) = M_{1,\mathrm{in}}^*.
\]
Choose sufficiently small neighborhoods $V_1^* \subset U_1^*$ and $V_2^* \subset U_2^*$ of $M_{1,\mathrm{out}}^*$ and $M_{2,\mathrm{out}}^*$, respectively, such that
\[
f^{N_{1 \to 2}}(V_1^*) \subset U_2^*,
\qquad
f^{N_{2 \to 1}}(V_2^*) \subset U_1^*.
\]
Define the global maps by
\begin{align*}
T_{1 \to 2}^* = T_{1 \to 2}^*(f, V_1^*) := f^{N_{1 \to 2}}|_{V_1^*},
\qquad
T_{2 \to 1}^* = T_{2 \to 1}^*(f, V_2^*) := f^{N_{2 \to 1}}|_{V_2^*}.
\end{align*}

\subsubsection{Geometrical conditions}
\label{sec-GC}

Since both hyperbolic periodic points have one-dimensional unstable directions,
\[
d_s := \dim W^s(O_1^*) = \dim W^s(O_2^*) = \dim(M_\mathrm{ph}) - 1.
\]
For each $\ell \in \{1,2\}$, let $E_\ell^\lambda$ and $E_\ell^\gamma$ be the eigenspaces of $Df^{\mathrm{per}(O_\ell^*)}(O_\ell^*)$ for the leading multipliers $\lambda_\ell^*$ and $\gamma_\ell^*$, and let $E_\ell^{ss}$ be the sum of the other stable eigenspaces. In $U_\ell^*$, fix a local strong stable foliation of a local stable manifold of $O_\ell^*$ and a two-dimensional local extended unstable invariant manifold through $O_\ell^*$, tangent at $O_\ell^*$ to $E_\ell^\lambda \oplus E_\ell^\gamma$; see \cite{SSTC2001}. For $P$ on this local stable manifold, let $\mathcal{F}^{ss}_\ell(P)$ be the local strong stable leaf through $P$. Any two such extended unstable manifolds are tangent along the local unstable manifold.

Near $M_{2 \to 1}^*$, $W^u(\mathrm{Orb}(O_2^*))$ is a curve and $W^s(\mathrm{Orb}(O_1^*))$ is a hypersurface. Let
\[
c = h \circ \zeta,
\]
where $\zeta : (-\delta,\delta) \longrightarrow W^u(\mathrm{Orb}(O_2^*))$ is a $C^r$ parametrization with $\zeta(0) = M_{2 \to 1}^*$ and $h$ is a $C^r$ defining function of $W^s(\mathrm{Orb}(O_1^*))$ near $M_{2 \to 1}^*$. We call the intersection at $M_{2 \to 1}^*$ a \emph{quadratic tangency} if
\[
c(0) = c'(0) = 0,
\qquad
c''(0) \neq 0.
\]
This condition is independent of the choices of $\zeta$ and $h$.

\begin{definition}
Let $W_1^{uE}$ and $W_2^{uE}$ be sufficiently small neighborhoods of $M_{1,\mathrm{out}}^*$ and $M_{2,\mathrm{out}}^*$ inside extended unstable manifolds of $O_1^*$ and $O_2^*$, respectively. We say that $(f,\Gamma^*)$ satisfies \textup{(P1)}--\textup{(P3)} if the following hold (see also Figure~\ref{fig-gc}):
\begin{itemize}
\item \textup{(P1)} $W^u(\mathrm{Orb}(O_2^*))$ and $W^s(\mathrm{Orb}(O_1^*))$ have a quadratic tangency at $M_{2 \to 1}^*$.
\item \textup{(P2)} $T_{1 \to 2}^*(W_1^{uE})$ is transverse to $\mathcal{F}^{ss}_2(M_{2,\mathrm{in}}^*)$ at $M_{2,\mathrm{in}}^*$, namely,
\[
T_{M_{2,\mathrm{in}}^*}\bigl(T_{1 \to 2}^*(W_1^{uE})\bigr) + T_{M_{2,\mathrm{in}}^*}\mathcal{F}^{ss}_2(M_{2,\mathrm{in}}^*) = T_{M_{2,\mathrm{in}}^*} M_\mathrm{ph}.
\]
\item \textup{(P3)} $T_{2 \to 1}^*(W_2^{uE})$ is transverse to $\mathcal{F}^{ss}_1(M_{1,\mathrm{in}}^*)$ at $M_{1,\mathrm{in}}^*$, namely,
\[
T_{M_{1,\mathrm{in}}^*}\bigl(T_{2 \to 1}^*(W_2^{uE})\bigr) + T_{M_{1,\mathrm{in}}^*}\mathcal{F}^{ss}_1(M_{1,\mathrm{in}}^*) = T_{M_{1,\mathrm{in}}^*} M_\mathrm{ph}.
\]
\end{itemize}
\end{definition}

\begin{figure}[h]
\centering
\includegraphics[width=0.54\linewidth]{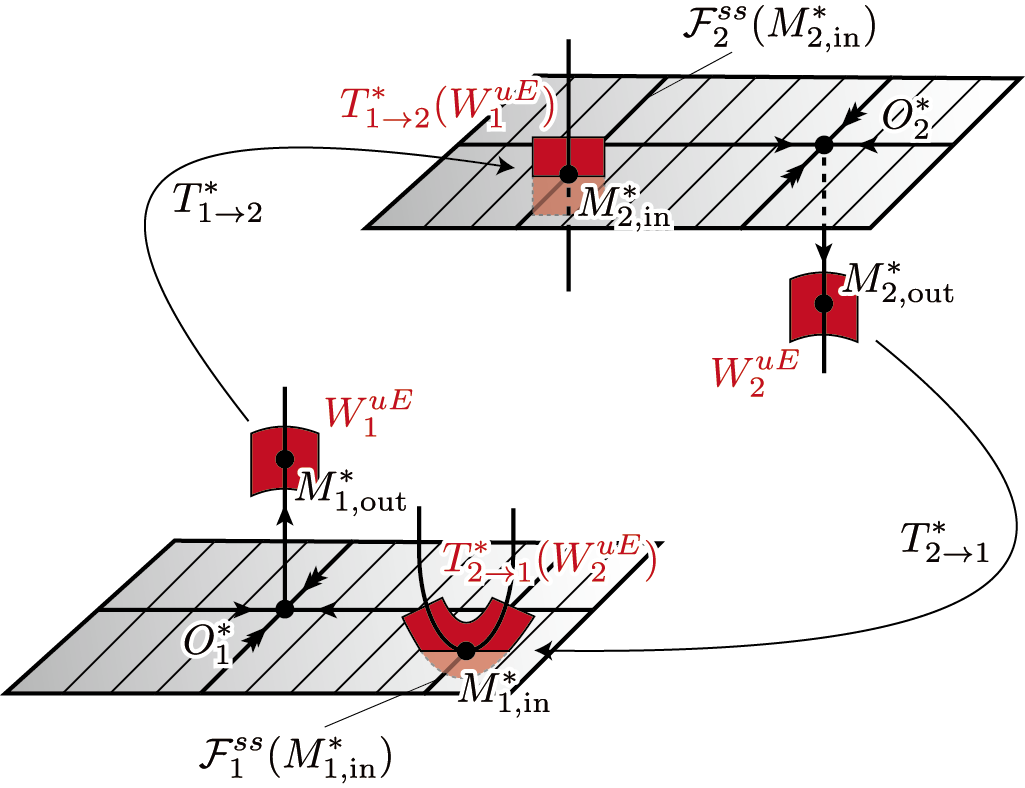}
\caption{Geometrical configuration of \textup{(P2)}, \textup{(P3)}}
\label{fig-gc}
\end{figure}

\begin{remark}
\label{rem-GCinv}
Because extended unstable manifolds are tangent along the local unstable manifold, \textup{(P2)} and \textup{(P3)} do not depend on the chosen manifolds or on smaller neighborhoods $W_\ell^{uE}$. Since $f$ is a diffeomorphism, moving the marked points along the heteroclinic connections does not change \textup{(P1)}--\textup{(P3)}. Hence \textup{(P1)}--\textup{(P3)} are finite-jet conditions at the marked points.
\end{remark}

Since \textup{(P1)}--\textup{(P3)} are finite-jet conditions at the marked points, one can verify the following proposition.

\begin{proposition}
\label{prop-generic}
Suppose that $f$ has a centrally dissipative-expanding transversal and non-transversal heteroclinic cycle $\Gamma^*$ of type two bi-saddles with one-dimensional unstable directions. Then every $C^r$ neighborhood of $f$ in $\mathrm{Diff}^r(M_\mathrm{ph})$ contains a diffeomorphism $g$ with a centrally dissipative-expanding transversal and non-transversal heteroclinic cycle $\Gamma_{\mathrm{new}}^*$ of the same type, associated with the hyperbolic continuations of $\mathrm{Orb}(O_1^*)$ and $\mathrm{Orb}(O_2^*)$, such that $(g,\Gamma_{\mathrm{new}}^*)$ satisfies \textup{(P1)}--\textup{(P3)}.
\end{proposition}

\subsection{Two Parameters}
\label{sec-TP}

Let $\{f_\varepsilon\}_{\varepsilon \in R_\mathrm{prm}^*} \subset \mathrm{Diff}^r(M_\mathrm{ph})$ be a two-parameter family with $f_{\varepsilon^*} = f$, where $R_\mathrm{prm}^* \subset \mathbb{R}^2$ is a connected open neighborhood of $\varepsilon^*$. We assume that the parameter dependence is of class $C^r$, namely that the map
\[
M_\mathrm{ph} \times R_\mathrm{prm}^* \longrightarrow M_\mathrm{ph},
\qquad
(P,\varepsilon) \longmapsto f_\varepsilon(P),
\]
is $C^r$. Fix a sufficiently small open disk $R_\mathrm{prm} \subset R_\mathrm{prm}^*$ centered at $\varepsilon^*$ such that each $O_\ell^*$ has a unique hyperbolic continuation $O_\ell(\varepsilon)$ for $\varepsilon \in R_\mathrm{prm}$.

Assume that $(f,\Gamma^*)$ satisfies \textup{(P1)}. Let $W^s \subset W^s(\mathrm{Orb}(O_1^*))$ and $W^u \subset W^u(\mathrm{Orb}(O_2^*))$ be small neighborhoods of $M_{2 \to 1}^*$. After shrinking $R_\mathrm{prm}$, these sets have continuations $W^s_\varepsilon \subset W^s(\mathrm{Orb}(O_1(\varepsilon)))$ and $W^u_\varepsilon \subset W^u(\mathrm{Orb}(O_2(\varepsilon)))$ that depend $C^r$ on $\varepsilon$. Choose the \emph{splitting parameter}, namely a $C^r$ function $\mu_1$ on $R_\mathrm{prm}$ given by the signed critical value of a local defining function of $W^s_\varepsilon$ restricted to $W^u_\varepsilon$ near $M_{2 \to 1}^*$.

Let $\lambda_\ell(\varepsilon)$ and $\gamma_\ell(\varepsilon)$ be the continuations of $\lambda_\ell^*$ and $\gamma_\ell^*$. Note that $\lambda_\ell(\varepsilon)$ and $\gamma_\ell(\varepsilon)$ are $C^{r - 1}$. Set
\[
\sigma_\ell(\varepsilon) := |\lambda_\ell(\varepsilon)\gamma_\ell(\varepsilon)|
\qquad
(\ell \in \{1,2\}).
\]
Since the cycle is centrally dissipative-expanding,
\[
\sigma_1^* < 1 < \sigma_2^*
\qquad\text{or}\qquad
\sigma_2^* < 1 < \sigma_1^*.
\]
After shrinking $R_\mathrm{prm}$, the same inequalities hold for all $\varepsilon \in R_\mathrm{prm}$. Define
\[
\mu_2(\varepsilon) := \frac{\log \sigma_2(\varepsilon)}{\log \sigma_1(\varepsilon)}.
\]
Then $\mu_2$ is $C^{r - 1}$,
\[
\mu_2(\varepsilon^*) = \frac{\log \sigma_2^*}{\log \sigma_1^*} < 0,
\]
and
\[
\sigma_1(\varepsilon)^i \sigma_2(\varepsilon)^j = \sigma_1(\varepsilon)^{i + \mu_2(\varepsilon)j}
\qquad
(i,j \in \mathbb{Z}).
\]

\begin{definition}
We say that the family $\{f_\varepsilon\}_{\varepsilon \in R_\mathrm{prm}^*}$ \emph{unfolds properly at $\varepsilon = \varepsilon^*$ with respect to $\Gamma^*$} if
\[
\det \frac{\partial(\mu_1,\mu_2)}{\partial\varepsilon}(\varepsilon^*) \neq 0.
\]
A family satisfying this condition will be called a \emph{proper unfolding family} of $(f,\Gamma^*)$.
\end{definition}

\begin{remark}
\label{rem-prm}
Under \textup{(P1)}--\textup{(P3)}, local perturbations with disjoint supports near $M_{2 \to 1}^*$ and near one periodic point produce a proper unfolding family of $(f,\Gamma^*)$.
\end{remark}

\begin{convention}
\label{conv-prm}
Assume $\{f_\varepsilon\}_{\varepsilon \in R_\mathrm{prm}^*}$ unfolds properly. After shrinking $R_\mathrm{prm}$ if necessary,
\[
R_\mathrm{prm} \ni \varepsilon \longmapsto \bigl(\mu_1(\varepsilon),\mu_2(\varepsilon)\bigr)
\]
is a $C^{r - 1}$ diffeomorphism onto its image. Henceforth we sometimes use $(\mu_1,\mu_2)$ as parameter and write $\varepsilon = (\mu_1,\mu_2)$. In particular,
\[
\varepsilon^* = \left(0,\frac{\log \sigma_2^*}{\log \sigma_1^*}\right).
\]
\end{convention}

\subsection{Local maps}

\subsubsection{Approximately linearized coordinates}

For the rest of Section~\ref{sec-prelim}, assume $(f,\Gamma^*)$ satisfies \textup{(P1)}--\textup{(P3)}, and $\{f_\varepsilon\}_{\varepsilon \in R_\mathrm{prm}^*}$ is a proper unfolding family with $f_{\varepsilon^*} = f$.

For each $\ell \in \{1,2\}$, let
\[
T_\ell = T_\ell(\varepsilon; f, U_\ell^*, \{ f_\varepsilon \}_{\varepsilon \in R_\mathrm{prm}^*}) := f_\varepsilon^{\mathrm{per}(O_\ell^*)}|_{U_\ell^*},
\qquad
\varepsilon \in R_\mathrm{prm}.
\]
Let $\lambda_{\ell,d_s}(\varepsilon),\lambda_{\ell,d_s-1}(\varepsilon),\dots,\lambda_{\ell,2}(\varepsilon)$ be the continuations of the remaining stable multipliers of $O_\ell^*$. Since $O_\ell(\varepsilon)$ is a hyperbolic fixed point of $T_\ell$ with one-dimensional weak stable and unstable directions and a $(d_s - 1)$-dimensional strong stable complement, \cite[Lemma~6]{GST2008} yields the following normal form.

\begin{lemma}[{\cite[Lemma~6]{GST2008}, adapted}]
\label{lem-ALC}
After shrinking $R_\mathrm{prm}$ if necessary, for every $\varepsilon \in R_\mathrm{prm}$ there exist $C^r$ coordinates $(x_\ell,y_\ell,u_\ell)$ on $U_\ell^*$, depending on $\varepsilon$, with $x_\ell,y_\ell \in \mathbb{R}$ and $u_\ell \in \mathbb{R}^{d_s-1}$, such that
\[
T_\ell : (x_\ell,y_\ell,u_\ell) \longmapsto (\widehat{x}_\ell,\widehat{y}_\ell,\widehat{u}_\ell)
\]
has the form
\begin{align*}
\widehat{x}_\ell &= \lambda_\ell x_\ell + \mathsf{P}_1^{(\ell)}(x_\ell,y_\ell,u_\ell,\varepsilon), \\
\widehat{y}_\ell &= \gamma_\ell y_\ell + \mathsf{P}_2^{(\ell)}(x_\ell,y_\ell,u_\ell,\varepsilon), \\
\widehat{u}_\ell &= \mathbf{A}_\ell u_\ell + \mathsf{P}_3^{(\ell)}(x_\ell,y_\ell,u_\ell,\varepsilon),
\end{align*}
where $\mathbf{A}_\ell = \mathbf{A}_\ell(\varepsilon)$ is a $(d_s - 1)\times(d_s - 1)$ matrix whose eigenvalues are
\[
\lambda_{\ell,d_s},\ \lambda_{\ell,d_s-1},\ \dots,\ \lambda_{\ell,2}.
\]
Moreover,
\begin{align*}
\mathsf{P}_1^{(\ell)}(0,y_\ell,0,\varepsilon),\ \mathsf{P}_2^{(\ell)}(0,y_\ell,0,\varepsilon),\ \mathsf{P}_3^{(\ell)}(0,y_\ell,0,\varepsilon) &\equiv 0, \\
\mathsf{P}_1^{(\ell)}(x_\ell,0,u_\ell,\varepsilon),\ \mathsf{P}_2^{(\ell)}(x_\ell,0,u_\ell,\varepsilon) &\equiv 0, \\
\partial_{x_\ell} \mathsf{P}_1^{(\ell)}(0,y_\ell,0,\varepsilon),\ \partial_{x_\ell} \mathsf{P}_3^{(\ell)}(0,y_\ell,0,\varepsilon) &\equiv 0, \\
\partial_{y_\ell} \mathsf{P}_2^{(\ell)}(x_\ell,0,u_\ell,\varepsilon) &\equiv 0.
\end{align*}
\end{lemma}

We call such coordinates an \emph{approximately linearized coordinate system}. In these coordinates, $x_\ell$ is weak stable, $y_\ell$ is unstable, and $u_\ell$ is the strong stable coordinate.

\begin{remark}\leavevmode
\label{rem-ALCdetails}
\begin{enumerate}
\item[(1)] In an approximately linearized coordinate system,
\[
W^u_\mathrm{loc}(O_\ell(\varepsilon)) = \{x_\ell = 0,\ u_\ell = 0\},
\qquad
W^s_\mathrm{loc}(O_\ell(\varepsilon)) = \{y_\ell = 0\}.
\]
Inside $W^s_\mathrm{loc}(O_\ell(\varepsilon))$, the local strong stable leaves are the slices
\[
\{x_\ell = \mathrm{const},\ y_\ell = 0\}.
\]
\item[(2)] If $(x_\ell^*,y_\ell^*,u_\ell^*)$ is any $C^r$ coordinate system on $U_\ell^*$ that does not depend on $\varepsilon$, then the coordinate change
\[
(x_\ell,y_\ell,u_\ell,\varepsilon) \longmapsto (x_\ell^*,y_\ell^*,u_\ell^*)
\]
and its first and second partial derivatives with respect to $(x_\ell,y_\ell,u_\ell)$ are $C^{r-2}$ in $(x_\ell,y_\ell,u_\ell,\varepsilon)$; see the remarks following \cite[Lemma~6]{GST2008}.
\item[(3)] Consequently, each $\mathsf{P}_\imath^{(\ell)}$ is $C^{r-2}$ in $(x_\ell,y_\ell,u_\ell,\varepsilon)$, and the same holds for its first and second partial derivatives with respect to $(x_\ell,y_\ell,u_\ell)$; in particular, $\mathsf{P}_\imath^{(\ell)}(\cdot, \cdot, \cdot, \varepsilon)$ is $C^r$ for fixed $\varepsilon$. Indeed, if $\psi$ is an $\varepsilon$-independent chart and $\varphi_\varepsilon$ is the chart from Lemma~\ref{lem-ALC}, then
\[
\varphi_\varepsilon \circ T_\ell \circ \varphi_\varepsilon^{-1}
=
(\varphi_\varepsilon \circ \psi^{-1})
\circ
(\psi \circ T_\ell \circ \psi^{-1})
\circ
(\varphi_\varepsilon \circ \psi^{-1})^{-1}.
\]
The middle factor is $C^r$, and item~(2) controls the two coordinate changes. Subtracting the linear part gives the claim.
\end{enumerate}
\end{remark}

\begin{convention}
Fix approximately linearized coordinate systems on $U_1^*$ and $U_2^*$, and write
\[
\mathfrak{F}_\mathrm{obj}
:=
\bigl(
f,\Gamma^*,(U_1^*;x_1,y_1,u_1),(U_2^*;x_2,y_2,u_2),
M_{1,\mathrm{in}}^*,M_{1,\mathrm{out}}^*,M_{2,\mathrm{in}}^*,M_{2,\mathrm{out}}^*,
V_1^*, V_2^*,
\{f_\varepsilon\}_{\varepsilon \in R_\mathrm{prm}^*}
\bigr).
\]
Unless stated otherwise, whenever
\begin{itemize}
\item we shrink $R_\mathrm{prm}$ or $\delta_\mathrm{dom}$ (see Section~\ref{sec-dom}) or
\item enlarge $\kappa_0$ (see the next section),
\end{itemize}
the new quantities depend only on $\mathfrak{F}_\mathrm{obj}$.
\end{convention}

\subsubsection{Representation of the iterated local maps}

Fix $\ell \in \{1,2\}$ and $k \in \mathbb{Z}_{>0}$. Consider an orbit segment
\[
(x_{\ell m},y_{\ell m},u_{\ell m}) = T_\ell^m(x_{\ell 0},y_{\ell 0},u_{\ell 0}),
\qquad
0 \le m \le k,
\]
that stays in $U_\ell^*$. Consider
\[
(x_{\ell 0},y_{\ell k},u_{\ell 0})
\]
for such an orbit segment. Let $\mathcal{A}_{\ell,k}(\varepsilon)$ be the set of all such triples. Set
\[
\widehat{\mathcal{A}}_{\ell,k}
:=
\{(x_{\ell 0},y_{\ell k},u_{\ell 0},\varepsilon) : \varepsilon \in R_\mathrm{prm},\ (x_{\ell 0},y_{\ell k},u_{\ell 0}) \in \mathcal{A}_{\ell,k}(\varepsilon)\}.
\]

\begin{lemma}[{\cite[Lemma~7]{GST2008}, adapted}]
\label{lem-ILMs}
After shrinking $R_\mathrm{prm}$ if necessary, there exist $\kappa_0 = \kappa_0(\mathfrak{F}_\mathrm{obj}) \in \mathbb{Z}_{>0}$, $\widehat{\lambda}_1$, $\widehat{\lambda}_2$, $\widehat{\gamma}_1$, and $\widehat{\gamma}_2$ such that, for every $k > \kappa_0$, every $\ell \in \{1,2\}$, and every $\varepsilon \in R_\mathrm{prm}$, every orbit segment of length $k$ in $U_\ell^*$ satisfies
\begin{align*}
x_{\ell k}
&=
\lambda_\ell^k x_{\ell 0}
+
\widehat{\lambda}_\ell^k
\mathsf{B}_{1,k}^{(\ell)}(x_{\ell 0},y_{\ell k},u_{\ell 0},\varepsilon), \\
y_{\ell 0}
&=
\gamma_\ell^{-k} y_{\ell k}
+
\widehat{\gamma}_\ell^{-k}
\mathsf{B}_{2,k}^{(\ell)}(x_{\ell 0},y_{\ell k},u_{\ell 0},\varepsilon), \\
u_{\ell k}
&=
\widehat{\lambda}_\ell^k
\mathsf{B}_{3,k}^{(\ell)}(x_{\ell 0},y_{\ell k},u_{\ell 0},\varepsilon).
\end{align*}
Here each $\mathsf{B}_{\imath,k}^{(\ell)}$ ($\imath \in \{1, 2, 3\}$) is defined on $\widehat{\mathcal{A}}_{\ell,k}$; each $\mathsf{B}_{\imath,k}^{(\ell)}$ is $C^{r-2}$ in $(x_{\ell 0},y_{\ell k},u_{\ell 0},\varepsilon)$, and the same holds for its first and second partial derivatives with respect to $(x_{\ell 0},y_{\ell k},u_{\ell 0})$ (In particular, $\mathsf{B}_{\imath,k}^{(\ell)}(\cdot, \cdot, \cdot, \varepsilon)$ is $C^r$ for fixed $\varepsilon$); for any finite integers $s$ and $l$ with $2 \le s \le r$ and $0 \le l \le r$, the norms
\[
\|\mathsf{B}_{\imath,k}^{(\ell)}\|_{C^{s-2}(\widehat{\mathcal{A}}_{\ell,k})}, \qquad \|\mathsf{B}_{\imath,k}^{(\ell)}(\cdot, \cdot, \cdot, \varepsilon)\|_{C^{l}(\mathcal{A}_{\ell,k}(\varepsilon))}
\]
are bounded uniformly in $k$ and $\varepsilon$; Moreover,
\[
\widehat{\lambda}_\ell < |\lambda_\ell(\varepsilon)|,
\qquad
\widehat{\gamma}_\ell > |\gamma_\ell(\varepsilon)|
\]
for every $\varepsilon \in R_\mathrm{prm}$.
\end{lemma}

\begin{remark}
\label{rem-hatssm}
Lemma~\ref{lem-ILMs} remains valid with $\widehat{\lambda}_\ell \, (< |\lambda_\ell^*|)$ replaced by a larger one $\widehat{\widehat{\lambda}}_\ell \, (< |\lambda_\ell^*|)$ and $\widehat{\gamma}_\ell \, (> |\gamma_\ell^*|)$ replaced by a smaller one $\widehat{\widehat{\gamma}}_\ell \, (> |\gamma_\ell^*|)$ ($\mathsf{B}_{\imath,k}^{(\ell)}$ may be changed, but it is still bounded). We sometimes replace them and rewrite $\widehat{\lambda}_\ell$ and $\widehat{\gamma}_\ell$ again. Set
\[
\lambda_\mathrm{max}(\varepsilon)
:=
\max\{|\lambda_1(\varepsilon)|,|\lambda_2(\varepsilon)|,|\gamma_1(\varepsilon)^{-1}|,|\gamma_2(\varepsilon)^{-1}|\}.
\]
Enlarging $\widehat{\lambda}_\ell$, reducing $\widehat{\gamma}_\ell$, and shrinking $R_\mathrm{prm}$, we may assume that
\begin{equation}
\label{eq-hineq}
|\lambda_1| \lambda_\mathrm{max} < \widehat{\lambda}_1,\quad
|\lambda_2| \lambda_\mathrm{max} < \widehat{\lambda}_2,\quad
|\gamma_1^{-1}| \lambda_\mathrm{max} < \widehat{\gamma}_1^{-1},\quad
|\gamma_2^{-1}| \lambda_\mathrm{max} < \widehat{\gamma}_2^{-1}
\end{equation}
for every $\varepsilon \in R_\mathrm{prm}$, where $\lambda_\ell = \lambda_\ell(\varepsilon)$ and $\gamma_\ell = \gamma_\ell(\varepsilon)$.
\end{remark}
\subsection{Global maps}

\subsubsection{Continuations of base points}

For $\varepsilon \in R_\mathrm{prm}$, we define the global maps by
\begin{align*}
T_{1 \to 2} = T_{1 \to 2}(\varepsilon; \mathfrak{F}_\mathrm{obj}) := f_\varepsilon^{N_{1 \to 2}}|_{V_1^*},
\qquad
T_{2 \to 1} = T_{2 \to 1}(\varepsilon; \mathfrak{F}_\mathrm{obj}) := f_\varepsilon^{N_{2 \to 1}}|_{V_2^*}.
\end{align*}
After shrinking $R_\mathrm{prm}$ if necessary, $T_{1 \to 2}(V_1^*) \subset U_2^*$ and $T_{2 \to 1}(V_2^*) \subset U_1^*$ for $\varepsilon \in R_\mathrm{prm}$. In approximately linearized coordinates, their coordinate expressions on these neighborhoods and their first and second partial derivatives with respect to $(x_1,y_1,u_1)$ or $(x_2,y_2,u_2)$ are $C^{r-2}$ in these variables and $\varepsilon$; see Remark~\ref{rem-ALCdetails}~(2), (3).

For $y_2$ near the $y_2$-coordinate $y_2^*$ of $M_{2,\mathrm{out}}^*$, write
\[
\Phi(y_2,\varepsilon)
:=
\operatorname{pr}_{y_1}\bigl(T_{2 \to 1}(0,y_2,0)\bigr).
\]

\begin{lemma}
\label{lem-Ms}
After shrinking $R_\mathrm{prm}$ if necessary, there exist unique families of points
\[
M_{1,\mathrm{in}}(\varepsilon),\ M_{1,\mathrm{out}}(\varepsilon) \in U_1^*,
\qquad
M_{2,\mathrm{in}}(\varepsilon),\ M_{2,\mathrm{out}}(\varepsilon) \in U_2^*
\]
of class $C^{r-2}$ in $\varepsilon$ such that:
\begin{itemize}
\item $M_{1,\mathrm{out}}(\varepsilon) \in W^u_\mathrm{loc}(O_1(\varepsilon))$,
\[
M_{2,\mathrm{in}}(\varepsilon) = T_{1 \to 2}(M_{1,\mathrm{out}}(\varepsilon)) \in W^s_\mathrm{loc}(O_2(\varepsilon)),
\]
and
\[
M_{1,\mathrm{out}} = (0,y_1^\mathrm{out},0),
\qquad
M_{2,\mathrm{in}} = (x_2^\mathrm{in},0,u_2^\mathrm{in}).
\]
\item The function $\Phi(\cdot,\varepsilon)$ has a unique critical point $y_2^\mathrm{out}(\varepsilon)$ near $y_2^*$. Set
\[
M_{2,\mathrm{out}}(\varepsilon) := (0,y_2^\mathrm{out}(\varepsilon),0),
\qquad
T_{2 \to 1}(M_{2,\mathrm{out}}(\varepsilon)) = (x_1^\mathrm{in}(\varepsilon),y_1^\mathrm{in}(\varepsilon),u_1^\mathrm{in}(\varepsilon)),
\]
and define
\[
M_{1,\mathrm{in}}(\varepsilon) := (x_1^\mathrm{in}(\varepsilon),0,u_1^\mathrm{in}(\varepsilon)).
\]
Then
\[
y_1^\mathrm{in}(\varepsilon) = \Phi(y_2^\mathrm{out}(\varepsilon),\varepsilon).
\]
\item At $\varepsilon = \varepsilon^*$,
\[
M_{1,\mathrm{in}} = M_{1,\mathrm{in}}^*,
\quad
M_{1,\mathrm{out}} = M_{1,\mathrm{out}}^*,
\quad
M_{2,\mathrm{in}} = M_{2,\mathrm{in}}^*,
\quad
M_{2,\mathrm{out}} = M_{2,\mathrm{out}}^*.
\]
\end{itemize}
\end{lemma}

\begin{proof}
Since the heteroclinic connection from $\mathrm{Orb}(O_1^*)$ to $\mathrm{Orb}(O_2^*)$ is transverse at $M_{1 \to 2}^*$, after shrinking $R_\mathrm{prm}$ the manifolds $W^u(\mathrm{Orb}(O_1(\varepsilon)))$ and $W^s(\mathrm{Orb}(O_2(\varepsilon)))$ meet in a unique point near $M_{1 \to 2}^*$. Transporting that point along the orbit segment of length $N_{1 \to 2}$ defines unique points $M_{1,\mathrm{out}}(\varepsilon)$ and $M_{2,\mathrm{in}}(\varepsilon)$ such that
\[
M_{2,\mathrm{in}} = T_{1 \to 2}(M_{1,\mathrm{out}}).
\]
Remark~\ref{rem-ALCdetails}~(1) gives
\[
M_{1,\mathrm{out}} = (0,y_1^\mathrm{out},0),
\qquad
M_{2,\mathrm{in}} = (x_2^\mathrm{in},0,u_2^\mathrm{in}).
\]

For the branch through $M_{2,\mathrm{out}}^*$, $\Phi$ is $C^{r-2}$. By \textup{(P1)} and Remark~\ref{rem-GCinv},
\[
\partial_{y_2}\Phi(y_2^*,\varepsilon^*) = 0,
\qquad
\partial_{y_2}^2\Phi(y_2^*,\varepsilon^*) \neq 0.
\]
Hence the implicit function theorem yields a unique $C^{r-2}$ function $y_2^\mathrm{out}(\varepsilon)$ near $y_2^*$ such that
\[
\partial_{y_2}\Phi(y_2^\mathrm{out}(\varepsilon),\varepsilon) = 0.
\]
Set
\[
M_{2,\mathrm{out}}(\varepsilon) = (0,y_2^\mathrm{out}(\varepsilon),0),
\qquad
T_{2 \to 1}(M_{2,\mathrm{out}}(\varepsilon)) = (x_1^\mathrm{in}(\varepsilon),y_1^\mathrm{in}(\varepsilon),u_1^\mathrm{in}(\varepsilon)),
\]
and define
\[
M_{1,\mathrm{in}}(\varepsilon) := (x_1^\mathrm{in}(\varepsilon),0,u_1^\mathrm{in}(\varepsilon)).
\]
At $\varepsilon = \varepsilon^*$, the transverse intersection and the critical point are the starred ones, so
\[
M_{1,\mathrm{in}} = M_{1,\mathrm{in}}^*,
\quad
M_{1,\mathrm{out}} = M_{1,\mathrm{out}}^*,
\quad
M_{2,\mathrm{in}} = M_{2,\mathrm{in}}^*,
\quad
M_{2,\mathrm{out}} = M_{2,\mathrm{out}}^*.
\]
Thus the four families are of class $C^{r-2}$ in $\varepsilon$.
\end{proof}

% \begin{remark}
% For $\varepsilon \neq \varepsilon^*$, only $M_{1,\mathrm{out}}(\varepsilon)$ and $M_{2,\mathrm{in}}(\varepsilon)$ are genuine heteroclinic continuations. The points $M_{2,\mathrm{out}}(\varepsilon)$ and $M_{1,\mathrm{in}}(\varepsilon)$ are auxiliary normalization points defined from the curve $y_2 \longmapsto T_{2 \to 1}(0,y_2,0)$.
% \end{remark}

\subsubsection{Representation of the global maps}

Write
\[
T_{1 \to 2}(x_1,y_1,u_1) = (\overline{x}_2,\overline{y}_2,\overline{u}_2),
\qquad
T_{2 \to 1}(x_2,y_2,u_2) = (\overline{x}_1,\overline{y}_1,\overline{u}_1).
\]
At $M_{1,\mathrm{out}} = (0,y_1^\mathrm{out},0)$, define
\begin{align*}
A_{12} &= \left.\partial_{x_1}\overline{x}_2\right|_{M_{1,\mathrm{out}}}, &
B_{12} &= \left.\partial_{y_1}\overline{x}_2\right|_{M_{1,\mathrm{out}}}, &
C_{12} &= \left.\partial_{x_1}\overline{y}_2\right|_{M_{1,\mathrm{out}}}, &
D_{12} &= \left.\partial_{y_1}\overline{y}_2\right|_{M_{1,\mathrm{out}}}, \\
E_{12} &= \left.\partial_{x_1}\overline{u}_2\right|_{M_{1,\mathrm{out}}}, &
F_{12} &= \left.\partial_{y_1}\overline{u}_2\right|_{M_{1,\mathrm{out}}},
\end{align*}
and at $M_{2,\mathrm{out}} = (0,y_2^\mathrm{out},0)$ define
\begin{align*}
A_{21} &= \left.\partial_{x_2}\overline{x}_1\right|_{M_{2,\mathrm{out}}}, &
B_{21} &= \left.\partial_{y_2}\overline{x}_1\right|_{M_{2,\mathrm{out}}}, &
C_{21} &= \left.\partial_{x_2}\overline{y}_1\right|_{M_{2,\mathrm{out}}}, \\
E_{21} &= \left.\partial_{x_2}\overline{u}_1\right|_{M_{2,\mathrm{out}}}, &
F_{21} &= \left.\partial_{y_2}\overline{u}_1\right|_{M_{2,\mathrm{out}}}, &
G_{21} &= \frac12\left.\partial_{y_2}^2\overline{y}_1\right|_{M_{2,\mathrm{out}}}.
\end{align*}
Because $y_2^\mathrm{out}(\varepsilon)$ is a critical point of $\Phi(\cdot,\varepsilon)$,
\[
\left.\partial_{y_2}\overline{y}_1\right|_{M_{2,\mathrm{out}}} = 0.
\]
Set
\[
\mathbf{J}_{12}
:=
\begin{pmatrix}
A_{12} & B_{12} \\
C_{12} & D_{12}
\end{pmatrix},
\qquad
\mathbf{J}_{21}
:=
\begin{pmatrix}
A_{21} & B_{21} \\
C_{21} & 0
\end{pmatrix}
\]

By Taylor expansion at $M_{1,\mathrm{out}}$ and $M_{2,\mathrm{out}}$, together with Lemma~\ref{lem-Ms}, one obtains the following proposition.

\begin{proposition}
\label{prop-globalrepr}
For $(x_{1i},y_{1i},u_{1i}) \in V_1^*$ and $(\overline{x}_{2j},\overline{y}_{2j},\overline{u}_{2j}) \in V_2^*$, set
\[
(\overline{x}_{20},\overline{y}_{20},\overline{u}_{20})
:=
T_{1 \to 2}(x_{1i},y_{1i},u_{1i}),
\qquad
(\overline{\overline{x}}_{10},\overline{\overline{y}}_{10},\overline{\overline{u}}_{10})
:=
T_{2 \to 1}(\overline{x}_{2j},\overline{y}_{2j},\overline{u}_{2j}).
\]
Then the following truncated Taylor representations hold:
\begin{equation}
\label{eq-T1to2}
\begin{aligned}
\overline{x}_{20} - x_2^\mathrm{in}
&=
A_{12} x_{1i}
+
B_{12}(y_{1i} - y_1^\mathrm{out})
+
\mathsf{R}^{(1)}_1, \\
\overline{y}_{20}
&=
C_{12} x_{1i}
+
D_{12}(y_{1i} - y_1^\mathrm{out})
+
\mathsf{R}^{(1)}_2, \\
\overline{u}_{20} - u_2^\mathrm{in}
&=
E_{12} x_{1i}
+
F_{12}(y_{1i} - y_1^\mathrm{out})
+
\mathsf{R}^{(1)}_3,
\end{aligned}
\end{equation}
and
\begin{equation}
\label{eq-T2to1}
\begin{aligned}
\overline{\overline{x}}_{10} - x_1^\mathrm{in}
&=
A_{21}\overline{x}_{2j}
+
B_{21}(\overline{y}_{2j} - y_2^\mathrm{out})
+
\mathsf{R}^{(2)}_1, \\
\overline{\overline{y}}_{10}
&=
y_1^\mathrm{in}
+
C_{21}\overline{x}_{2j}
+
G_{21}(\overline{y}_{2j} - y_2^\mathrm{out})^2
+
\mathsf{R}^{(2)}_2, \\
\overline{\overline{u}}_{10} - u_1^\mathrm{in}
&=
E_{21}\overline{x}_{2j}
+
F_{21}(\overline{y}_{2j} - y_2^\mathrm{out})
+
\mathsf{R}^{(2)}_3,
\end{aligned}
\end{equation}
where
\begin{align*}
\mathsf{R}^{(1)}_\ell
&=
\mathsf{R}^{(1)}_\ell(x_{1i},y_{1i} - y_1^\mathrm{out},u_{1i},\varepsilon)
=
O\!\left(\|(x_{1i},y_{1i} - y_1^\mathrm{out})\|^2 + \|u_{1i}\|\right),
\qquad
\ell \in \{1,2,3\}, \\
\mathsf{R}^{(2)}_\imath
&=
\mathsf{R}^{(2)}_\imath(\overline{x}_{2j},\overline{y}_{2j} - y_2^\mathrm{out},\overline{u}_{2j},\varepsilon)
=
O\!\left(\|(\overline{x}_{2j},\overline{y}_{2j} - y_2^\mathrm{out})\|^2 + \|\overline{u}_{2j}\|\right),
\qquad
\imath \in \{1,3\}, \\
\mathsf{R}^{(2)}_2
&=
\mathsf{R}^{(2)}_2(\overline{x}_{2j},\overline{y}_{2j} - y_2^\mathrm{out},\overline{u}_{2j},\varepsilon) \\
&=
O\!\left(
|\overline{x}_{2j}|^2
+
|\overline{x}_{2j}(\overline{y}_{2j} - y_2^\mathrm{out})|
+
|\overline{y}_{2j} - y_2^\mathrm{out}|^3
+
\|\overline{u}_{2j}\|
\right).
\end{align*}
Here $\|\cdot\|$ denotes the Euclidean norm and all coefficients are $C^{r-2}$ in $\varepsilon$. Moreover, the remainder terms are $C^{r-2}$ and the same holds for their first and second partial derivatives with respect to the variables except for $\varepsilon$; in particular, they are $C^r$ for fixed $\varepsilon$.
\end{proposition}

\begin{remark}
\label{rem-gcmeaning}
After shrinking $R_\mathrm{prm}$ if necessary, the geometric conditions admit the following algebraic interpretation.
\begin{itemize}
\item In approximately linearized coordinates, the strong stable direction is the $u_\ell$-axis, so $\mathbf{J}_{12}$ and $\mathbf{J}_{21}$ are the Jacobian matrices of the $(x_2,y_2)$- and $(x_1,y_1)$-components of $T_{1 \to 2}|_{W_1^{uE}}$ and $T_{2 \to 1}|_{W_2^{uE}}$.
\item Condition \textup{(P1)} implies $G_{21} \neq 0$. Indeed, the restriction of \eqref{eq-T2to1} to $(\overline{x}_{2j},\overline{u}_{2j}) = (0,0)$ is the Taylor expansion of $y_2 \longmapsto \operatorname{pr}_{y_1}(T_{2 \to 1}(0,y_2,0))$ at its critical point, and quadratic tangency means that its quadratic coefficient is nonzero.
\item Condition \textup{(P2)} is equivalent to
\[
J_{12} := \det \mathbf{J}_{12} = A_{12}D_{12} - B_{12}C_{12} \neq 0,
\]
because $T_{1 \to 2}(W_1^{uE})$ is transverse to the strong stable leaf at $M_{2,\mathrm{in}}$ exactly when the Jacobian matrix of the $(x_2,y_2)$-component of $T_{1 \to 2}|_{W_1^{uE}}$ is invertible.
\item Condition \textup{(P3)} is equivalent to $B_{21}C_{21} \neq 0$, because
\[
\det \mathbf{J}_{21} = -B_{21}C_{21},
\]
and the same argument applies at $M_{1,\mathrm{in}}$.
\item The transverse intersection on the heteroclinic connection from $\mathrm{Orb}(O_1)$ to $\mathrm{Orb}(O_2)$ implies $D_{12} \neq 0$, since along the local unstable branch $\{x_1 = 0,\ u_1 = 0\}$ the target stable manifold is $\{y_2 = 0\}$.
\end{itemize}
\end{remark}

\begin{convention}
\label{conv-repmu1}
Restricting \eqref{eq-T2to1} to $(\overline{x}_{2j},\overline{u}_{2j}) = (0,0)$ gives
\[
\Phi(y_2,\varepsilon)
=
y_1^\mathrm{in}(\varepsilon)
+
G_{21}(\varepsilon)(y_2 - y_2^\mathrm{out}(\varepsilon))^2
+
O\bigl(|y_2 - y_2^\mathrm{out}(\varepsilon)|^3\bigr).
\]
Since $G_{21}(\varepsilon) \neq 0$, the tangency condition is $y_1^\mathrm{in}(\varepsilon) = 0$. After shrinking $R_\mathrm{prm}$ if necessary, we use $(y_1^\mathrm{in},\mu_2)$ as a parameter and rename $y_1^\mathrm{in}$ as $\mu_1$.
\end{convention}
\subsection{First-return maps}

\subsubsection{The definition of the first-return maps}
For $i,j > \kappa_0$, let
\begin{equation}
\label{eq-5pts}
\begin{aligned}
&(x_{10},y_{10},u_{10}),\ (\overline{\overline{x}}_{10},\overline{\overline{y}}_{10},\overline{\overline{u}}_{10}) \in U_1^*, \quad (x_{1i},y_{1i},u_{1i}) \in V_1^* \\
&(\overline{x}_{20},\overline{y}_{20},\overline{u}_{20}) \in U_2^*, \quad (\overline{x}_{2j},\overline{y}_{2j},\overline{u}_{2j}) \in V_2^*
\end{aligned}
\end{equation}
and assume
\begin{equation}
\label{eq-5ptseq}
\begin{aligned}
(x_{1m},y_{1m},u_{1m})
&=
T_1^m(x_{10},y_{10},u_{10}) \in U_1^*, \qquad 0 \le m \le i, \\
(\overline{x}_{20},\overline{y}_{20},\overline{u}_{20})
&=
T_{1 \to 2}(x_{1i},y_{1i},u_{1i}), \\
(\overline{x}_{2m},\overline{y}_{2m},\overline{u}_{2m})
&=
T_2^m(\overline{x}_{20},\overline{y}_{20},\overline{u}_{20}) \in U_2^* \qquad 0 \le m \le j, \\
(\overline{\overline{x}}_{10},\overline{\overline{y}}_{10},\overline{\overline{u}}_{10})
&=
T_{2 \to 1}(\overline{x}_{2j},\overline{y}_{2j},\overline{u}_{2j}).
\end{aligned}
\end{equation}

\begin{definition}
We set
\[
\mathrm{Dom}(T_{ij})
:=
\{
(x_{10},y_{10},u_{10}) \in U_1^*
\mid
\text{there exist points satisfying \eqref{eq-5pts} and \eqref{eq-5ptseq}}
\}.
\]
For $(x_{10},y_{10},u_{10}) \in \mathrm{Dom}(T_{ij})$, define
\[
T_{ij}(x_{10},y_{10},u_{10})
:=
(\overline{\overline{x}}_{10},\overline{\overline{y}}_{10},\overline{\overline{u}}_{10}).
\]
Equivalently,
\[
T_{ij} = T_{2 \to 1} \circ T_2^j \circ T_{1 \to 2} \circ T_1^i.
\]
\end{definition}

\subsubsection{Existence of the domains of the first-return maps}
\label{sec-dom}

For $\delta_\mathrm{dom} > 0$ and $\ell \in \{1,2\}$, set
\begin{align*}
B_\ell^\mathrm{in}
&:=
\{
(x_\ell,u_\ell)
\mid
|x_\ell - x_\ell^\mathrm{in}| \le \delta_\mathrm{dom},
\|u_\ell - u_\ell^\mathrm{in}\| \le \delta_\mathrm{dom}
\},
\\
I_\ell^\mathrm{in}
&:=
\{
y_\ell
\mid
|y_\ell| \le \delta_\mathrm{dom}
\},
\\
B_\ell^\mathrm{out}
&:=
\{
(x_\ell,u_\ell)
\mid
|x_\ell| \le \delta_\mathrm{dom},
\|u_\ell\| \le \delta_\mathrm{dom}
\},
\\
I_\ell^\mathrm{out}
&:=
\{
y_\ell
\mid
|y_\ell - y_\ell^\mathrm{out}| \le \delta_\mathrm{dom}
\}.
\end{align*}
Set
\[
\Pi_\ell^\mathrm{in} := B_\ell^\mathrm{in} \times I_\ell^\mathrm{in},
\qquad
\Pi_\ell^\mathrm{out} := B_\ell^\mathrm{out} \times I_\ell^\mathrm{out}.
\]
Take small $\delta_\mathrm{dom} = \delta_\mathrm{dom}(\mathfrak{F}_\mathrm{obj}) > 0$ so that the four boxes $\Pi_1^\mathrm{in}$, $\Pi_1^\mathrm{out}$, $\Pi_2^\mathrm{in}$, and $\Pi_2^\mathrm{out}$ are pairwise disjoint and satisfy
\[
\Pi_\ell^\mathrm{in} \subset U_\ell^*, \quad
\Pi_\ell^\mathrm{out} \subset V_\ell^*
\qquad
(\ell \in \{1,2\}).
\]
After shrinking $R_\mathrm{prm}$ if necessary, this implies
\begin{equation}
\label{eq-Pi12asmp}
T_{1 \to 2}(\Pi_1^\mathrm{out}) \subset U_2^*,
\qquad
T_{2 \to 1}(\Pi_2^\mathrm{out}) \subset U_1^*.
\end{equation}

\begin{definition}
A subset $\Lambda_\ell \subset \Pi_\ell^\mathrm{in}$ is a \emph{horizontal region} if
\[
\Lambda_\ell = \bigsqcup_{t \in [0,1]} \mathrm{graph}(\varphi_t)
\]
with $B_\ell^\mathrm{in} \times [0, 1] \ni (x_\ell,u_\ell,t) \longmapsto \varphi_t(x_\ell,u_\ell) \in I_\ell^\mathrm{in}$ of class $C^1$.
\end{definition}

\begin{definition}
For $k > \kappa_0$ and $\ell \in \{1,2\}$, set
\[
\Pi_{\ell,k}^\mathrm{in}(\varepsilon)
:=
\Pi_\ell^\mathrm{in} \cap T_\ell^{-k}(\Pi_\ell^\mathrm{out}),
\qquad
\Pi_{\ell,k}^\mathrm{out}(\varepsilon)
:=
T_\ell^k(\Pi_{\ell,k}^\mathrm{in}).
\]
After shrinking $R_\mathrm{prm}$ and $\delta_\mathrm{dom}$ and enlarging $\kappa_0$ if necessary, one can verify that $\Pi_{\ell,k}^\mathrm{in}$ is a horizontal region and $T_\ell^m(\Pi_{\ell,k}^\mathrm{in}) \subset U_\ell^*$ for $m \in \{0, \cdots, k\}$.
\end{definition}

\begin{definition}
For $i,j > \kappa_0$, define
\[
\widetilde{\mathrm{Dom}}(T_{ij})
:=
\{
P \in \mathrm{Dom}(T_{ij})
\mid
T_1^m \circ T_{ij}(P) \in U_1^* \text{ for any $m \in \{0, \cdots, i\}$}
\}.
\]
\end{definition}

\begin{lemma}
\label{lem-ptex}
After shrinking $R_\mathrm{prm}$ and enlarging $\kappa_0$ if necessary, the following hold for every $i,j > \kappa_0$.
\begin{itemize}
\item The domain $\mathrm{Dom}(T_{ij})$ contains a horizontal region $\Lambda_1^\mathrm{in} \subset \Pi_1^\mathrm{in}$; see Figure~\ref{fig-cycle-dom}.
\item There exists a constant $K_1 = K_1(\mathfrak{F}_\mathrm{obj}) > 0$ such that the hypothesis
\[
\left|\mu_1 + C_{21}x_2^\mathrm{in}\lambda_2^j - y_1^\mathrm{out} \gamma_1^{-i}\right|
<
K_1\bigl(|\lambda_1^i\lambda_2^j| + |\gamma_1|^{-i}\bigr)
\]
implies that $\widetilde{\mathrm{Dom}}(T_{ij})$ contains a closed set $\widetilde{\Lambda}_1^\mathrm{in} \subset \Lambda_1^\mathrm{in}$ with nonempty interior.
\end{itemize}
\end{lemma}

\begin{figure}[h]
\centering
\includegraphics[width=0.80\linewidth]{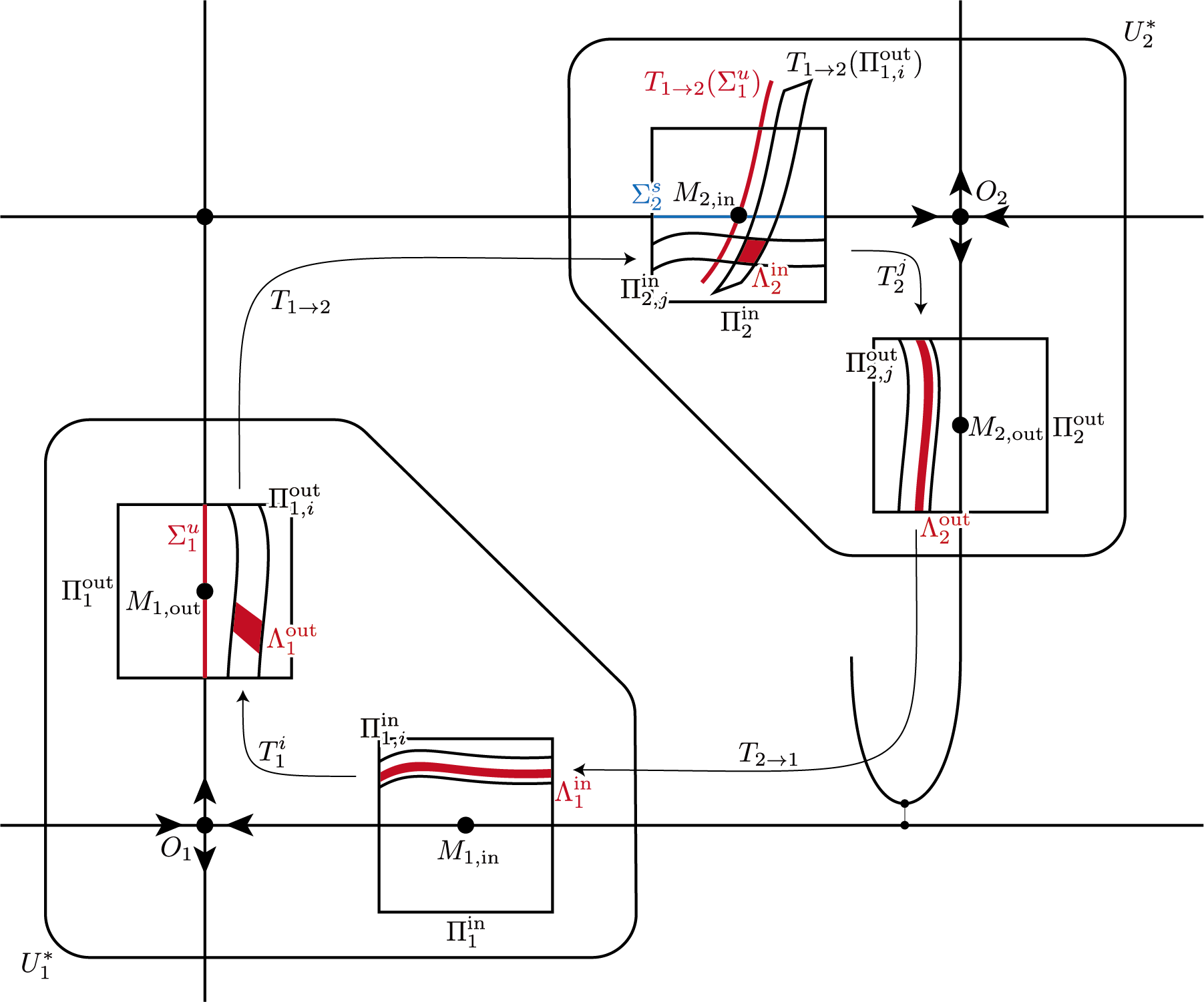}
\caption{Domain of the first-return map.}
\label{fig-cycle-dom}
\end{figure}

\begin{proof}
Set
\begin{align*}
\Sigma_1^u
&:=
\{x_1 = 0,\ u_1 = 0,\ |y_1 - y_1^\mathrm{out}| \le \delta_\mathrm{dom}\},
\\
\Sigma_2^s
&:=
\{\,|x_2 - x_2^\mathrm{in}| \le \delta_\mathrm{dom},\ \|u_2 - u_2^\mathrm{in}\| \le \delta_\mathrm{dom},\ y_2 = 0\,\}.
\end{align*}
For $(x_1^*,u_1^*) \in B_1^\mathrm{in}$ and $y_2^* \in I_2^\mathrm{out}$, set
\[
V_1^\mathrm{out}(x_1^*,u_1^*)
:=
T_1^i\bigl(\Pi_{1,i}^\mathrm{in} \cap \{x_1 = x_1^*,\ u_1 = u_1^*\}\bigr)
\subset
\Pi_{1,i}^\mathrm{out},
\]
\[
H_2^\mathrm{in}(y_2^*)
:=
T_2^{-j}\bigl(\Pi_{2,j}^\mathrm{out} \cap \{y_2 = y_2^*\}\bigr)
\subset
\Pi_{2,j}^\mathrm{in}.
\]
After enlarging $\kappa_0 = \kappa_0(\mathfrak{F}_\mathrm{obj})$ if necessary, the families $V_1^\mathrm{out}(x_1^*,u_1^*)$ and $H_2^\mathrm{in}(y_2^*)$ are uniformly $C^1$-close to $\Sigma_1^u$ and $\Sigma_2^s$. Since $T_{1 \to 2}(\Sigma_1^u)$ meets $\Sigma_2^s$ transversely at $M_{2,\mathrm{in}}$, after shrinking $R_\mathrm{prm}$ and enlarging $\kappa_0$ there is a connected component $\Lambda_2^\mathrm{in}$ of $T_{1 \to 2}(\Pi_{1,i}^\mathrm{out}) \cap \Pi_{2,j}^\mathrm{in}$ near $M_{2,\mathrm{in}}$ such that each curve $T_{1 \to 2}(V_1^\mathrm{out}(x_1^*,u_1^*))$ meets each surface $H_2^\mathrm{in}(y_2^*)$ transversely in one point of $\Lambda_2^\mathrm{in}$. In particular, $\Lambda_2^\mathrm{in}$ is closed.

Set
\[
\Lambda_1^\mathrm{out} := T_{1 \to 2}^{-1}(\Lambda_2^\mathrm{in}),
\qquad
\Lambda_1^\mathrm{in} := T_1^{-i}(\Lambda_1^\mathrm{out}),
\qquad
\Lambda_2^\mathrm{out} := T_2^j(\Lambda_2^\mathrm{in}).
\]
Then $\Lambda_1^\mathrm{out}$ and $\Lambda_2^\mathrm{out}$ are closed, $\Lambda_1^\mathrm{in}$ is a horizontal region, and
\[
T_{2 \to 1}(\Lambda_2^\mathrm{out}) \subset U_1^*
\]
by \eqref{eq-Pi12asmp}.
Therefore $\Lambda_1^\mathrm{in} \subset \mathrm{Dom}(T_{ij})$. This proves item~1.

Since $U_1^*$ contains both $O_1(\varepsilon)$ and $M_{1,\mathrm{out}}(\varepsilon)$, Lemma~\ref{lem-ILMs} implies that, after shrinking $R_\mathrm{prm}$ if necessary, there exists
\[
K_2 = K_2(\mathfrak{F}_\mathrm{obj}) > \sup_{\varepsilon \in R_\mathrm{prm}} |y_1^\mathrm{out}(\varepsilon)|
\]
such that
\[
\bigcap_{m = 0}^i T_1^{-m}(U_1^*)
\supset
\{\,|y_1| < K_2|\gamma_1|^{-i}\,\}
\]
for every $i > \kappa_0$.

Choose $K_3 > 0$ small. For $|s - y_2^\mathrm{out}| \le K_3$, set
\[
S_2^\mathrm{out}(s)
:=
\Lambda_2^\mathrm{out} \cap \{y_2 = s\}.
\]
Then
\[
S_2^\mathrm{out}(y_2^\mathrm{out})
=
T_2^j\bigl(\Lambda_2^\mathrm{in} \cap H_2^\mathrm{in}(y_2^\mathrm{out})\bigr).
\]
Its $\overline{x}_{2j}$-projection has length comparable to $|\lambda_1^i\lambda_2^j|$. Therefore, using \eqref{eq-T2to1}, Remark~\ref{rem-gcmeaning}, and Convention~\ref{conv-repmu1}, after shrinking $R_\mathrm{prm}$ and enlarging $\kappa_0$ there exists $K_4 = K_4(\mathfrak{F}_\mathrm{obj}) > 0$ such that
\[
\bigl(
\mu_1 + C_{21}x_2^\mathrm{in}\lambda_2^j - K_4|\lambda_1^i\lambda_2^j|,
\mu_1 + C_{21}x_2^\mathrm{in}\lambda_2^j + K_4|\lambda_1^i\lambda_2^j|
\bigr)
\subset
(\mathrm{pr}_{y_1} \circ T_{2 \to 1})(S_2^\mathrm{out}(y_2^\mathrm{out}))
\]
where $\mathrm{pr}_{y_1}(x_1,y_1,u_1) := y_1$. After reducing $K_3$ if necessary, the same inclusion holds for every $s$ with $|s - y_2^\mathrm{out}| \le K_3$.

Choose $K_1 > 0$ so that
\[
K_1 \le K_4
\qquad\text{and}\qquad
\sup_{\varepsilon \in R_\mathrm{prm}} |y_1^\mathrm{out}(\varepsilon)| + K_1 \le K_2.
\]
If
\[
\left|\mu_1 + C_{21}x_2^\mathrm{in}\lambda_2^j - y_1^\mathrm{out}\gamma_1^{-i}\right|
<
K_1\bigl(|\lambda_1^i\lambda_2^j| + |\gamma_1|^{-i}\bigr),
\]
then
\[
\begin{aligned}
\left|\mu_1 + C_{21}x_2^\mathrm{in}\lambda_2^j\right|
&\le
\left|\mu_1 + C_{21}x_2^\mathrm{in}\lambda_2^j - y_1^\mathrm{out}\gamma_1^{-i}\right|
+
|y_1^\mathrm{out}|\,|\gamma_1|^{-i} \\
&<
K_1|\lambda_1^i\lambda_2^j|
+
\bigl(K_1 + |y_1^\mathrm{out}|\bigr)|\gamma_1|^{-i} \\
&\le
K_4|\lambda_1^i\lambda_2^j| + K_2|\gamma_1|^{-i}.
\end{aligned}
\]
Therefore every set $S_2^\mathrm{out}(s)$ with $|s - y_2^\mathrm{out}| \le K_3$ meets the strip $\{|y_1| < K_2|\gamma_1|^{-i}\}$. Hence
\[
\Omega_2^\mathrm{out}
:=
\Lambda_2^\mathrm{out}
\cap
\{|y_2 - y_2^\mathrm{out}| \le K_3\}
\cap
T_{2 \to 1}^{-1}(\{|y_1| < K_2|\gamma_1|^{-i}\})
\]
has nonempty interior. Choose a closed set $\widetilde{\Lambda}_2^\mathrm{out} \subset \Omega_2^\mathrm{out}$ with nonempty interior.

Set
\[
\widetilde{\Lambda}_2^\mathrm{in}
:=
T_2^{-j}(\widetilde{\Lambda}_2^\mathrm{out}) \subset \Lambda_2^\mathrm{in},
\qquad
\widetilde{\Lambda}_1^\mathrm{out}
:=
T_{1 \to 2}^{-1}(\widetilde{\Lambda}_2^\mathrm{in}) \subset \Lambda_1^\mathrm{out},
\]
and
\[
\widetilde{\Lambda}_1^\mathrm{in}
:=
T_1^{-i}(\widetilde{\Lambda}_1^\mathrm{out}) \subset \Lambda_1^\mathrm{in}.
\]
Since the three displayed maps are diffeomorphisms, $\widetilde{\Lambda}_1^\mathrm{in}$ is closed and has nonempty interior. By construction,
\[
T_{ij}(\widetilde{\Lambda}_1^\mathrm{in})
\subset
T_{2 \to 1}(\widetilde{\Lambda}_2^\mathrm{out})
\subset
\bigcap_{m = 0}^i T_1^{-m}(U_1^*),
\]
hence $\widetilde{\Lambda}_1^\mathrm{in} \subset \widetilde{\mathrm{Dom}}(T_{ij})$. This proves item~2.
\end{proof}

\section{The rescaling lemma and its proof}
\label{sec-rescaling}

\subsection{The rescaling lemma}

Set
\[
\mu_2^* := \frac{\log \sigma_2^*}{\log \sigma_1^*} < 0,
\qquad
\mathcal{I} := \left\{ (i,j) \in \mathbb{Z}_{>\kappa_0}^2 \mid |i + \mu_2^* j| \le 1 \right\}.
\]
Note that $\mathcal{I}$ is infinite. Because if $\alpha := -\mu_2^* \in \mathbb{Q}$, take all sufficiently large exact solutions of $i = \alpha j$; if $\alpha \notin \mathbb{Q}$, take $i$ to be a nearest integer to $\alpha j$ for infinitely many $j$.

\begin{convention}
\label{not-bigO}
Let
\[
F = F(\varepsilon,i,j,P_1,P_2),
\qquad
G = G(\varepsilon,i,j,P_1,P_2),
\]
where
\[
\varepsilon \in R_\mathrm{prm},
\qquad
(i,j) \in \mathcal{I},
\qquad
P_1 \in U_1^*,
\qquad
P_2 \in U_2^*.
\]
We write $F = O(G)$ if
\[
|F| \le C\,|G|
\]
for all arguments with a constant $C = C(\mathfrak{F}_\mathrm{obj}) > 0$. If $G$ does not vanish, then $F = o(G)$ means
\[
\sup_{\varepsilon \in R_\mathrm{prm},\, P_1 \in U_1^*,\, P_2 \in U_2^*}
\left| \frac{F(\varepsilon,i,j,P_1,P_2)}{G(\varepsilon,i,j,P_1,P_2)} \right|
\xrightarrow[\min\{i,j\}\to\infty,\ (i,j)\in\mathcal{I}]{} 0.
\]
Unless stated otherwise, every occurrence of $O(\cdot)$ or $o(\cdot)$ in Section~\ref{sec-rescaling} is understood in this uniform sense.
\end{convention}

\begin{lemma}[Rescaling lemma]
\label{lem-rescaling}
After enlarging $\kappa_0$ if necessary, there exists a family $\{\Delta_{ij}\}_{(i,j)\in\mathcal I}$, where each $\Delta_{ij} \subset R_\mathrm{prm}$ is a connected open set, such that
\[
\sup_{\varepsilon \in \Delta_{ij}} \|\varepsilon - \varepsilon^*\|
\xrightarrow[\min\{i,j\}\to\infty,\ (i,j)\in\mathcal I]{} 0.
\]
Moreover, for every $(i,j) \in \mathcal{I}$ and every $\varepsilon \in \Delta_{ij}$ there exist $C^r$ coordinates $(X,Y,U)$ on $\widetilde{\Lambda}_1^\mathrm{in}$ depending on $(\varepsilon,i,j)$ such that the first-return map
\[
T_{ij} : (X,Y,U) \longmapsto (\overline{X},\overline{Y},\overline{U})
\]
has the form
\begin{align*}
\overline{X} &= Y, \\
\overline{Y} &= M_1 - M_2 X - Y^2 + \mathsf{Y}_{ij}(X,Y,U,M_1,M_2), \\
\overline{U} &= \mathsf{U}_{ij}(X,Y,U,M_1,M_2),
\end{align*}
where $X,Y \in \mathbb{R}, \, U \in \mathbb{R}^{d_s-1}$, and
\begin{align*}
M_1 &= -D_{12}^2 G_{21} \left( \mu_1 + C_{21} x_2^\mathrm{in} \lambda_2^j - y_1^\mathrm{out} \gamma_1^{-i} \right) \gamma_1^{2i} \gamma_2^{2j}, \\
M_2 &= -B_{21} C_{21} J_{12} \lambda_1^i \gamma_1^i \lambda_2^j \gamma_2^j.
\end{align*}
The map
\[
\Delta_{ij} \ni \varepsilon \longmapsto (M_1(\varepsilon),M_2(\varepsilon)) \in \widetilde{\Delta}_{ij}
\]
is a $C^{r-2}$ diffeomorphism onto
\[
\widetilde{\Delta}_{ij} := (-2,4) \times (\varsigma_{ij} - 1/2, \varsigma_{ij} + 1/2),
\]
where
\[
\varsigma_{ij}
:=
\operatorname{sign}\left(
\left.
-B_{21}C_{21}J_{12}\lambda_1^i\gamma_1^i\lambda_2^j\gamma_2^j
\right|_{\varepsilon=\varepsilon^*}
\right)
\in \{-1,1\}.
\]
The coordinates $(X, Y, U)$ depend on $\varepsilon$ with the same regularity as in Remark~\ref{rem-ALCdetails}~(2). The rescaled domains are asymptotically large: every compact set $K \subset \mathbb{R}^{\dim(M_\mathrm{ph})}$ is contained in the rescaled domain for all sufficiently large $(i,j) \in \mathcal{I}$. Finally, for any finite integers $s$ and $l$ with $2 \le s \le r$ and $0 \le l \le r$ and every bounded sets
\[
B \subset \mathbb{R}^{\dim(M_\mathrm{ph})}, \quad \widehat{B} := B \times \widetilde{\Delta}_{ij}
\]
the norms satisfy
\[
\|\mathsf{Y}_{ij}\|_{C^{s-2}(\widehat{B})},\ \|\mathsf{U}_{ij}\|_{C^{s-2}(\widehat{B})},\ \|\mathsf{Y}_{ij}(\cdot, \cdot, \cdot, M_1, M_2)\|_{C^{l}(B)},\ \|\mathsf{U}_{ij}(\cdot, \cdot, \cdot, M_1, M_2)\|_{C^{l}(B)} = o(1).
\]
\end{lemma}
\subsection{Proof of the rescaling lemma}

\subsubsection{Compositions}

Take points \eqref{eq-5pts} satisfying \eqref{eq-5ptseq}, with $(x_{10},y_{10},u_{10}) \in \widetilde{\Lambda}_1^\mathrm{in}$, and set
\[
(\overline{\overline{x}}_{1i},\overline{\overline{y}}_{1i},\overline{\overline{u}}_{1i})
:=
T_1^i(\overline{\overline{x}}_{10},\overline{\overline{y}}_{10},\overline{\overline{u}}_{10}).
\]
Lemma~\ref{lem-ILMs}, together with \eqref{eq-T1to2}, \eqref{eq-T2to1}, and \eqref{eq-hineq}, gives
\begin{equation}
\label{eq-FLGcomp}
\begin{aligned}
\overline{x}_{20} - x_2^\mathrm{in}
&=
A_{12}\lambda_1^i x_{10}
+ B_{12}(y_{1i} - y_1^\mathrm{out})
+
\mathsf{L}_1^{(1)},
\\
\gamma_2^{-j}\overline{y}_{2j}
+
\widehat{\gamma}_2^{-j}
\mathsf{B}_{2,j}^{(2)}(\overline{x}_{20},\overline{y}_{2j},\overline{u}_{20},\varepsilon)
&=
C_{12}\lambda_1^i x_{10}
+
D_{12}(y_{1i} - y_1^\mathrm{out})
+
\mathsf{L}_2^{(1)},
\\
\overline{u}_{20} - u_2^\mathrm{in}
&=
E_{12}\lambda_1^i x_{10}
+
F_{12}(y_{1i} - y_1^\mathrm{out})
+
\mathsf{L}_3^{(1)},
\end{aligned}
\end{equation}
where, for $\ell \in \{1,2,3\}$,
\[
\mathsf{L}_\ell^{(1)}(x_{10}, y_{1i} - y_1^\mathrm{out}, u_{10}, \varepsilon)
=
O\!\left(
\widehat{\lambda}_1^i
+
|\lambda_1^i x_{10}(y_{1i} - y_1^\mathrm{out})|
+
|y_{1i} - y_1^\mathrm{out}|^2
\right).
\]
Likewise,
\begin{equation}
\label{eq-SLGcomp}
\begin{aligned}
\overline{\overline{x}}_{10} - x_1^\mathrm{in}
&=
A_{21}\lambda_2^j\overline{x}_{20}
+
B_{21}(\overline{y}_{2j} - y_2^\mathrm{out})
+
\mathsf{L}_1^{(2)},
\\
\gamma_1^{-i}\overline{\overline{y}}_{1i}
+
\widehat{\gamma}_1^{-i}
\mathsf{B}_{2,i}^{(1)}(\overline{\overline{x}}_{10},\overline{\overline{y}}_{1i},\overline{\overline{u}}_{10},\varepsilon)
&=
\mu_1
+
C_{21}\lambda_2^j\overline{x}_{20}
+
G_{21}(\overline{y}_{2j} - y_2^\mathrm{out})^2
+
\mathsf{L}_2^{(2)},
\\
\overline{\overline{u}}_{10} - u_1^\mathrm{in}
&=
E_{21}\lambda_2^j\overline{x}_{20}
+
F_{21}(\overline{y}_{2j} - y_2^\mathrm{out})
+
\mathsf{L}_3^{(2)},
\end{aligned}
\end{equation}
where
\begin{align*}
\mathsf{L}_\imath^{(2)}(\overline{x}_{20}, \overline{y}_{2j} - y_2^\mathrm{out}, \overline{u}_{20}, \varepsilon)
&=
O\!\left(
\widehat{\lambda}_2^j
+
|\lambda_2^j \overline{x}_{20}(\overline{y}_{2j} - y_2^\mathrm{out})|
+
|\overline{y}_{2j} - y_2^\mathrm{out}|^2
\right),
\qquad
\imath \in \{1,3\},
\\
\mathsf{L}_2^{(2)}(\overline{x}_{20}, \overline{y}_{2j} - y_2^\mathrm{out}, \overline{u}_{20}, \varepsilon)
&=
O\!\left(
\widehat{\lambda}_2^j
+
|\lambda_2^j \overline{x}_{20}(\overline{y}_{2j} - y_2^\mathrm{out})|
+
|\overline{y}_{2j} - y_2^\mathrm{out}|^3
\right).
\end{align*}

\subsubsection{Shilnikov coordinates}

For points satisfying \eqref{eq-5pts} and \eqref{eq-5ptseq}, set
\begin{equation}
\label{eq-shilcoords}
\begin{aligned}
\xi_1 := x_{10} - x_1^\mathrm{in},
\qquad
\eta_1 := y_{1i} - y_1^\mathrm{out},
\qquad
\upsilon_1 := u_{10} - u_1^\mathrm{in},
\\
\overline{\xi}_2 := \overline{x}_{20} - x_2^\mathrm{in},
\qquad
\overline{\eta}_2 := \overline{y}_{2j} - y_2^\mathrm{out},
\qquad
\overline{\upsilon}_2 := \overline{u}_{20} - u_2^\mathrm{in},
\\
\overline{\overline{\xi}}_1 := \overline{\overline{x}}_{10} - x_1^\mathrm{in},
\qquad
\overline{\overline{\eta}}_1 := \overline{\overline{y}}_{1i} - y_1^\mathrm{out},
\qquad
\overline{\overline{\upsilon}}_1 := \overline{\overline{u}}_{10} - u_1^\mathrm{in}.
\end{aligned}
\end{equation}
These are the so-called \emph{Shilnikov coordinates}.

Lemma~\ref{lem-ILMs} gives
\[
|x_{1i}| + \|u_{1i}\| + |\overline{y}_{20}|
=
O\bigl(|\lambda_1^i| + |\gamma_2^{-j}|\bigr)
\]
on $\Lambda_1^\mathrm{in}$. Hence
\begin{equation}
\label{eq-domain-bd}
|y_{1i} - y_1^\mathrm{out}|
+
|\overline{x}_{20} - x_2^\mathrm{in}|
+
\|\overline{u}_{20} - u_2^\mathrm{in}\|
=
O\bigl(|\lambda_1^i| + |\gamma_2^{-j}|\bigr)
\end{equation}
on $\Lambda_1^\mathrm{in}$, by \eqref{eq-T1to2}, Remark~\ref{rem-gcmeaning}, and the smallness of $\delta_\mathrm{dom}$. Therefore, for $(x_{10},y_{10},u_{10}) \in \widetilde{\Lambda}_1^\mathrm{in} \subset \Lambda_1^\mathrm{in}$,
\begin{equation}
\label{eq-shilrest}
\xi_1,\ \upsilon_1,\ \overline{\eta}_2
=
O(1),
\qquad
\eta_1,\ \overline{\xi}_2,\ \overline{\upsilon}_2
=
O\bigl(|\lambda_1^i| + |\gamma_2^{-j}|\bigr).
\end{equation}

Using \eqref{eq-shilcoords}, \eqref{eq-FLGcomp}, \eqref{eq-SLGcomp}, and \eqref{eq-shilrest}, Taylor expansion gives
\begin{equation}
\label{eq-FLGshil}
\begin{aligned}
\overline{\xi}_2
&=
A_{12}\lambda_1^i x_1^\mathrm{in}[1+o(1)]
+
A_{12}\lambda_1^i \xi_1[1+o(1)]
+
B_{12}\eta_1[1+o(1)]
+
\mathsf{R}_1^{(3)},
\\
\gamma_2^{-j}\overline{\eta}_2
+
\widehat{\gamma}_2^{-j}\mathsf{B}_{2,j}^{(3)}
&=
\bigl(C_{12}\lambda_1^i x_1^\mathrm{in} - \gamma_2^{-j}y_2^\mathrm{out}\bigr)[1+o(1)]
+
C_{12}\lambda_1^i \xi_1[1+o(1)]
+
D_{12}\eta_1[1+o(1)]
+
\mathsf{R}_2^{(3)},
\\
\overline{\upsilon}_2
&=
E_{12}\lambda_1^i x_1^\mathrm{in}[1+o(1)]
+
E_{12}\lambda_1^i \xi_1[1+o(1)]
+
F_{12}\eta_1[1+o(1)]
+
\mathsf{R}_3^{(3)},
\end{aligned}
\end{equation}
where
\begin{equation}
\label{eq-FLGshilsub}
\begin{aligned}
\mathsf{R}_\ell^{(3)}(\xi_1,\eta_1,\upsilon_1,\varepsilon)
&=
O(\widehat{\lambda}_1^i + \widehat{\gamma}_2^{-j}),
\qquad
\ell \in \{1,2,3\},
\\
\mathsf{B}_{2,j}^{(3)}(\overline{\xi}_2,\overline{\eta}_2,\overline{\upsilon}_2,\varepsilon)
&:=
\mathsf{B}_{2,j}^{(2)}(\overline{\xi}_2 + x_2^\mathrm{in},\overline{\eta}_2 + y_2^\mathrm{out},\overline{\upsilon}_2 + u_2^\mathrm{in},\varepsilon).
\end{aligned}
\end{equation}
Similarly,
\begin{equation}
\label{eq-SLGshil}
\begin{aligned}
\overline{\overline{\xi}}_1
&=
A_{21}\lambda_2^j x_2^\mathrm{in}[1+o(1)]
+
A_{21}\lambda_2^j \overline{\xi}_2[1+o(1)]
+
B_{21}\overline{\eta}_2[1+o(1)]
+
\mathsf{R}_1^{(4)},
\\
\gamma_1^{-i}\overline{\overline{\eta}}_1
+
\widehat{\gamma}_1^{-i}\mathsf{B}_{2,i}^{(4)}
&=
\bigl(
\mu_1
+
C_{21}\lambda_2^j x_2^\mathrm{in}
-
\gamma_1^{-i}y_1^\mathrm{out}
\bigr)[1+o(1)]
+
C_{21}\lambda_2^j \overline{\xi}_2[1+o(1)]
+
G_{21}\overline{\eta}_2^2[1+o(1)]
+
\mathsf{R}_2^{(4)},
\\
\overline{\overline{\upsilon}}_1
&=
E_{21}\lambda_2^j x_2^\mathrm{in}[1+o(1)]
+
E_{21}\lambda_2^j \overline{\xi}_2[1+o(1)]
+
F_{21}\overline{\eta}_2[1+o(1)]
+
\mathsf{R}_3^{(4)},
\end{aligned}
\end{equation}
where
\begin{equation}
\label{eq-SLGshilsub}
\begin{aligned}
\mathsf{R}_\imath^{(4)}(\overline{\xi}_2,\overline{\eta}_2,\overline{\upsilon}_2,\varepsilon)
&=
O(\widehat{\lambda}_2^j + |\overline{\eta}_2|^2),
\qquad
\imath \in \{1,3\},
\\
\mathsf{R}_2^{(4)}(\overline{\xi}_2,\overline{\eta}_2,\overline{\upsilon}_2,\varepsilon)
&=
O(\widehat{\lambda}_2^j + |\overline{\eta}_2|^3),
\\
\mathsf{B}_{2,i}^{(4)}(\overline{\overline{\xi}}_1,\overline{\overline{\eta}}_1,\overline{\overline{\upsilon}}_1,\varepsilon)
&:=
\mathsf{B}_{2,i}^{(1)}(\overline{\overline{\xi}}_1 + x_1^\mathrm{in},\overline{\overline{\eta}}_1 + y_1^\mathrm{out},\overline{\overline{\upsilon}}_1 + u_1^\mathrm{in},\varepsilon).
\end{aligned}
\end{equation}

\subsubsection{First translation}

\begin{definition}
\label{def-weighted-remainder}
Write
\[
z_1 = (\xi_1,\eta_1,\upsilon_1,\overline{\xi}_2,\overline{\eta}_2,\overline{\upsilon}_2), \quad z_2 = (\overline{\xi}_2,\overline{\eta}_2,\overline{\upsilon}_2,\overline{\overline{\xi}}_1,\overline{\overline{\eta}}_1,\overline{\overline{\upsilon}}_1).
\]
A family $\mathsf{W}_{ij}(z,\varepsilon)$ ($z \in \{z_1, z_2\}$) is of \emph{weighted remainder type} if, on every bounded set $B$ in the $z$-space and for any finite integers $s$ and $l$ with $2 \le s \le r$ and $0 \le l \le r$, it is a finite sum of terms
\[
a(z,\varepsilon)\rho, \quad a(z,\varepsilon)m(z),
\]
where $\rho \in \{\widehat{\lambda}_1^i,\widehat{\gamma}_2^{-j},\widehat{\lambda}_2^j,\widehat{\gamma}_1^{-i}\}$, $m(z)$ is a monomial  in the components of $z$ of total degree at least $2$, and
\[
\|a\|_{C^{s-2}(B \times R_\mathrm{prm})}, \, \|a(\cdot, \varepsilon)\|_{C^l(B)} = O(1).
\]
\end{definition}

% \begin{remark}
% \label{rem-weighted-remainder-calculus}
% Weighted remainders are preserved under finite sums, multiplication by uniformly bounded coefficients, triangular near-identity coordinate changes with coefficients $O(|\lambda_1^i| + |\gamma_1^{-i}| + |\lambda_2^j| + |\gamma_2^{-j}|)$, and local elimination of scalar equations $w + q\,\Phi(w,z,\varepsilon) = L(z,\varepsilon) + \mathsf{W}(z,\varepsilon)$ with $q = o(1)$, $\Phi = O(1)$, $L$ affine in $z$, and $\mathsf{W}$ of weighted remainder type. In the last case, the affine part is absorbed into $L[1+o(1)]$.
% \end{remark}

\begin{lemma}
\label{lem-shift-1}
After enlarging $\kappa_0$ if necessary, there exists a translation
\[
(\xi_1,\eta_1,\upsilon_1,\overline{\xi}_2,\overline{\eta}_2,\overline{\upsilon}_2)
\longmapsto
(\xi_1 - \xi_1^*,\eta_1 - \eta_1^*,\upsilon_1 - \upsilon_1^*,\overline{\xi}_2 - \overline{\xi}_2^*,\overline{\eta}_2 - \overline{\eta}_2^*,\overline{\upsilon}_2 - \overline{\upsilon}_2^*).
\]
with
\begin{equation}
\label{eq-shift1-est}
\begin{aligned}
\xi_1^*,\ \eta_1^*,\ \upsilon_1^*,\ \overline{\xi}_2^*,\ \overline{\eta}_2^*,\ \overline{\upsilon}_2^*
&=
O(|\lambda_1^i| + |\gamma_1^{-i}| + |\lambda_2^j| + |\gamma_2^{-j}|),
\end{aligned}
\end{equation}
such that, in the translated coordinates, again denoted by the same symbols,
\begin{equation}
\label{eq-FLGerase}
\begin{aligned}
\overline{\xi}_2
&=
A_{12}\lambda_1^i \xi_1
+
B_{12}\eta_1
+
\mathsf{W}_1^{(5)},
\\
\gamma_2^{-j}\overline{\eta}_2
&=
C_{12}\lambda_1^i \xi_1
+
D_{12}\eta_1
+
\mathsf{W}_2^{(5)},
\\
\overline{\upsilon}_2
&=
E_{12}\lambda_1^i \xi_1
+
F_{12}\eta_1
+
\mathsf{W}_3^{(5)},
\end{aligned}
\end{equation}
and
\begin{equation}
\label{eq-SLGerase}
\begin{aligned}
\overline{\overline{\xi}}_1
&=
A_{21}\lambda_2^j \overline{\xi}_2
+
B_{21}\overline{\eta}_2
+
\mathsf{W}_1^{(6)},
\\
\gamma_1^{-i} \overline{\overline{\eta}}_1
&=
\bigl(
\mu_1
+
C_{21} \lambda_2^j x_2^\mathrm{in}
-
\gamma_1^{-i} y_1^\mathrm{out}
\bigr)
+
C_{21} \lambda_2^j \overline{\xi}_2
+
G_{21} \overline{\eta}_2^2
+
\mathsf{W}_2^{(6)},
\\
\overline{\overline{\upsilon}}_1
&=
E_{21}\lambda_2^j \overline{\xi}_2
+
F_{21}\overline{\eta}_2
+
\mathsf{W}_3^{(6)},
\end{aligned}
\end{equation}
where $\mathsf{W}_\ell^{(5)}(z_1,\varepsilon)$ and $\mathsf{W}_\ell^{(6)}(z_2,\varepsilon)$, for $\ell \in \{1,2,3\}$, are of weighted remainder type in the sense of Definition~\ref{def-weighted-remainder}.
\end{lemma}

\begin{proof}
Let
\[
z^* := (\xi_1^*,\eta_1^*,\upsilon_1^*,\overline{\xi}_2^*,\overline{\eta}_2^*,\overline{\upsilon}_2^*)
\]
be the translation vector. We choose it so that the constants in \eqref{eq-FLGshil} and in the first and third equations of \eqref{eq-SLGshil} vanish, and so that the coefficient of $\overline{\eta}_2$ in the second equation of \eqref{eq-SLGshil} vanishes at the new origin.

These conditions define a self-map on $z^*$. The first and third equations of \eqref{eq-FLGshil} determine $\overline{\xi}_2^*$ and $\overline{\upsilon}_2^*$; since $D_{12} \neq 0$, the second equation determines $\eta_1^*$; the first and third equations of \eqref{eq-SLGshil} determine $\xi_1^*$ and $\upsilon_1^*$; and, because $G_{21} \neq 0$, the derivative condition for the second equation determines $\overline{\eta}_2^*$.

Using \eqref{eq-FLGshilsub}, \eqref{eq-SLGshilsub}, \eqref{eq-hineq}, and Lemma~\ref{lem-ILMs}, all constant terms have the size stated in \eqref{eq-shift1-est}, and every feedback coefficient is
\[
O(|\lambda_1^i| + |\gamma_1^{-i}| + |\lambda_2^j| + |\gamma_2^{-j}|).
\]
Hence the self-map sends the box from \eqref{eq-shift1-est} into itself and has Lipschitz constant $o(1)$. The contraction mapping theorem gives a unique fixed point $z^*$.

After translating by this fixed point, Taylor expansion keeps the displayed affine and quadratic terms explicit, and every other contribution is of weighted remainder type. In particular, the first-transition scalar relation is left in the implicit form used in \eqref{eq-FLGerase}. This gives \eqref{eq-FLGerase} and \eqref{eq-SLGerase}.
\end{proof}

\subsubsection{Triangularization}
Assume \eqref{eq-FLGerase}. Since
\[
J_{12} = A_{12}D_{12} - B_{12}C_{12} \neq 0,
\qquad
D_{12} \neq 0,
\]
by Remark~\ref{rem-gcmeaning}, apply the following successive triangular changes, dropping the superscript ``new'' after each step:
\begin{align*}
\eta_1^\mathrm{new}
&:= \eta_1 + \frac{C_{12}}{D_{12}}\lambda_1^i \xi_1,
\\
\overline{\xi}_2^\mathrm{new}
&:= \overline{\xi}_2 - \frac{B_{12}}{D_{12}}\gamma_2^{-j}\overline{\eta}_2,
\\
\overline{\upsilon}_2^\mathrm{new}
&:= \overline{\upsilon}_2
- \frac{E_{12}D_{12} - F_{12}C_{12}}{J_{12}}\overline{\xi}_2
- \frac{F_{12}}{D_{12}}\gamma_2^{-j}\overline{\eta}_2.
\end{align*}
They are well defined. One can verify the following lemma.

\begin{lemma}
\label{lem-triangularization}
In the new coordinates, again denoted by the same symbols,
\begin{equation}
\label{eq-FLGdiag3}
\begin{aligned}
\overline{\xi}_2
&=
\frac{J_{12}}{D_{12}}\lambda_1^i \xi_1
+
\mathsf{W}_1^{(7)},
\\
\gamma_2^{-j}\overline{\eta}_2
&=
D_{12}\eta_1
+
\mathsf{W}_2^{(7)},
\\
\overline{\upsilon}_2
&=
\mathsf{W}_3^{(7)}.
\end{aligned}
\end{equation}
and
\begin{equation}
\label{eq-SLGdiag}
\begin{aligned}
\overline{\overline{\xi}}_1
&=
A_{21} \lambda_2^j \overline{\xi}_2 + B_{21}\overline{\eta}_2
+
\mathsf{W}_1^{(8)},
\\
\gamma_1^{-i} \overline{\overline{\eta}}_1
&=
\bigl(
\mu_1
+
C_{21} \lambda_2^j x_2^\mathrm{in}
-
\gamma_1^{-i} y_1^\mathrm{out}
\bigr) \\
&\quad +
C_{21} \lambda_2^j \overline{\xi}_2
+
\left( \frac{B_{21} C_{12}}{D_{12}} \lambda_1^i \gamma_1^{-i}  + \frac{B_{12} C_{21}}{D_{12}} \lambda_2^j \gamma_2^{-j} \right) \overline{\eta}_2
+
G_{21} \overline{\eta}_2^2
+
\mathsf{W}_2^{(8)},
\\
\overline{\overline{\upsilon}}_1
&=
E_{21} \lambda_2^j \overline{\xi}_2 + F_{21}\overline{\eta}_2
+
\mathsf{W}_3^{(8)},
\end{aligned}
\end{equation}
where, for each $\ell \in \{1,2,3\}$, both $\mathsf{W}_\ell^{(7)}(z, \varepsilon)$ and $\mathsf{W}_\ell^{(8)}(z, \varepsilon)$ are of weighted remainder type.
\end{lemma}

\subsubsection{Second translation}
Arguing as in Lemma~\ref{lem-shift-1}, one obtains the following lemma.

\begin{lemma}
\label{lem-shift-2}
After enlarging $\kappa_0$ if necessary, there exists a second translation of size
\begin{align*}
\xi_1^*,\ \eta_1^*,\ \upsilon_1^*,\ \overline{\xi}_2^*,\ \overline{\eta}_2^*,\ \overline{\upsilon}_2^*
&=
O(|\lambda_1^i| + |\gamma_1^{-i}| + |\lambda_2^j| + |\gamma_2^{-j}|),
\end{align*}
such that, in the translated coordinates,
\begin{equation}
\label{eq-pre-rescale}
\begin{aligned}
\overline{\xi}_2 &= \frac{J_{12}}{D_{12}}\lambda_1^i \xi_1 + \mathsf{W}_1^{(9)},
\qquad
\gamma_2^{-j}\overline{\eta}_2 = D_{12}\eta_1 + \mathsf{W}_2^{(9)},
\qquad
\overline{\upsilon}_2 = \mathsf{W}_3^{(9)}, \\
\overline{\overline{\xi}}_1 &= A_{21} \lambda_2^j \overline{\xi}_2 + B_{21}\overline{\eta}_2 + \mathsf{W}_1^{(10)},
\\
\gamma_1^{-i}\overline{\overline{\eta}}_1
&=
\bigl( \mu_1 + C_{21} \lambda_2^j x_2^\mathrm{in} - \gamma_1^{-i}y_1^\mathrm{out} \bigr)
+
C_{21}\lambda_2^j \overline{\xi}_2
+
G_{21}\overline{\eta}_2^2
+
\mathsf{W}_2^{(10)},
\\
\overline{\overline{\upsilon}}_1 &= E_{21} \lambda_2^j \overline{\xi}_2 + F_{21}\overline{\eta}_2 + \mathsf{W}_3^{(10)}.
\end{aligned}
\end{equation}
Moreover, for each $\ell \in \{1,2,3\}$, both $\mathsf{W}_\ell^{(9)}(z_1,\varepsilon)$ and $\mathsf{W}_\ell^{(10)}(z_2,\varepsilon)$ are of weighted remainder type in the sense of Definition~\ref{def-weighted-remainder}.
\end{lemma}
\subsubsection{Rescaling}

Define the rescaled coordinates by
\begin{equation}
\label{eq-rescaledXY}
\begin{aligned}
\xi_1 &= -\frac{B_{21}}{G_{21}D_{12}}\gamma_1^{-i}\gamma_2^{-j} X_1,
\quad
\eta_1 = -\frac{1}{G_{21}D_{12}^2}\gamma_1^{-i}\gamma_2^{-2j} Y_1,
\quad
\upsilon_1 - \frac{F_{21}}{B_{21}}\xi_1 = -\frac{B_{21}}{G_{21}D_{12}}\gamma_1^{-i}\gamma_2^{-j} U_1.
\end{aligned}
\end{equation}
and
\begin{equation}
\label{eq-rescaledXY2}
\begin{aligned}
\overline{\xi}_2 &= -\frac{J_{12}B_{21}}{G_{21}D_{12}^2}\lambda_1^i\gamma_1^{-i}\gamma_2^{-j}\overline{X}_2,
\quad
\overline{\eta}_2 = -\frac{1}{G_{21}D_{12}}\gamma_1^{-i}\gamma_2^{-j}\overline{Y}_2,
\quad
\overline{\upsilon}_2 = -\frac{B_{21}}{G_{21}D_{12}}\gamma_1^{-i}\gamma_2^{-j}\overline{U}_2.
\end{aligned}
\end{equation}
For the image variables, use \eqref{eq-rescaledXY} for $(\overline{\overline{\xi}}_1,\overline{\overline{\eta}}_1,\overline{\overline{\upsilon}}_1)$ and denote the corresponding rescaled coordinates by $(\overline{\overline{X}}_1,\overline{\overline{Y}}_1,\overline{\overline{U}}_1)$. Also set
\begin{equation}
\label{eq-M1M2}
\begin{aligned}
M_1
&:=
-D_{12}^2 G_{21}
\left(
\mu_1
+
C_{21}x_2^\mathrm{in}\lambda_2^j
-
y_1^\mathrm{out}\gamma_1^{-i}
\right)
\gamma_1^{2i}\gamma_2^{2j},
\\
M_2
&:=
-B_{21}C_{21}J_{12}\lambda_1^i\gamma_1^i\lambda_2^j\gamma_2^j.
\end{aligned}
\end{equation}
Then, one can find the following lemma.

\begin{lemma}
\label{lem-rescaled-LG}
On every bounded set in the rescaled variables,
\begin{align*}
\overline{X}_2 &= X_1 + \mathsf{S}_1^{(11)}, \\
\overline{Y}_2 &= Y_1 + \mathsf{S}_2^{(11)}, \\
\overline{U}_2 &= \mathsf{S}_3^{(11)},
\end{align*}
and
\begin{align*}
\overline{\overline{X}}_1
&=
\frac{A_{21}J_{12}}{D_{12}}\lambda_1^i\lambda_2^j\overline{X}_2
+
\overline{Y}_2
+
\mathsf{S}_1^{(12)},
\\
\overline{\overline{Y}}_1
&=
M_1
-
M_2 \overline{X}_2
-
\overline{Y}_2^2
+
\mathsf{S}_2^{(12)},
\\
\overline{\overline{U}}_1
&=
\mathsf{S}_3^{(12)}.
\end{align*}
Here, for each $\ell \in \{1,2,3\}$, each bounded set $B_1$ in the
\[
Z_1 = (X_1, Y_1, U_1, \overline{X}_2, \overline{Y}_2, \overline{U}_2)
\]
space, each bounded set $B_2$ in the
\[
Z_2 = (\overline{X}_2, \overline{Y}_2, \overline{U}_2,\overline{\overline{X}}_1, \overline{\overline{Y}}_1, \overline{\overline{U}}_1)
\]
space, for any finite integers $s$ and $l$ with $2 \le s \le r$ and $0 \le l \le r$, $\mathsf{S}_\ell^{(11)}(Z_1, \varepsilon)$ and $\mathsf{S}_\ell^{(12)}(Z_2, \varepsilon)$ satisfy
\[
\|\mathsf{S}_\ell^{(11)}\|_{C^{s-2}(B_1 \times R_\mathrm{prm})}, \, \|\mathsf{S}_\ell^{(11)}(\cdot, \varepsilon)\|_{C^l(B_1)}, \, \|\mathsf{S}_\ell^{(12)}\|_{C^{s-2}(B_2 \times R_\mathrm{prm})}, \, \|\mathsf{S}_\ell^{(12)}(\cdot, \varepsilon)\|_{C^l(B_2)} = o(1).
\]
\end{lemma}

\subsubsection{Final composition}

Recall that in Lemma~\ref{lem-rescaled-LG} the terms $\mathsf{S}_\ell^{(11)}$, $\mathsf{S}_\ell^{(12)}$ depend on
\[
(X_1, Y_1, U_1, \overline{X}_2, \overline{Y}_2, \overline{U}_2), \quad (\overline{X}_2, \overline{Y}_2, \overline{U}_2,\overline{\overline{X}}_1, \overline{\overline{Y}}_1, \overline{\overline{U}}_1),
\]
respectively. On every bounded set of values $(M_1, M_2)$, one can resolve the first three equations with respect to $(\overline{X}_2, \overline{Y}_2, \overline{U}_2)$, and the resolved equations are still of the same form, with new terms satisfying the same $o(1)$ estimates. The same argument resolves the last three equations with respect to $(\overline{\overline{X}}_1, \overline{\overline{Y}}_1, \overline{\overline{U}}_1)$.

Substituting the first resolved equations into the second resolved equations, we obtain
\begin{align*}
\overline{\overline{X}}_1 &= Y_1 + \mathsf{S}_1^{(13)}, \\
\overline{\overline{Y}}_1 &= M_1 - M_2 X_1 - Y_1^2 + \mathsf{S}_2^{(13)}, \\
\overline{\overline{U}}_1 &= \mathsf{S}_3^{(13)},
\end{align*}
where, for each $\ell \in \{1,2,3\}$, $\mathsf{S}_\ell^{(13)}(X_1, Y_1, U_1)$ satisfies the same $o(1)$ estimates.

Finally, set
\[
X := X_1,
\qquad
Y := Y_1 + \mathsf{S}_1^{(13)},
\qquad
U := U_1,
\]
and write
\[
(\overline{X},\overline{Y},\overline{U})
:=
(\overline{\overline{X}}_1,\overline{\overline{Y}}_1,\overline{\overline{U}}_1).
\]
Then the first-return map takes the form
\begin{align*}
\overline{X} &= Y, \\
\overline{Y} &= M_1 - M_2 X - Y^2 + \mathsf{Y}_{ij}(X,Y,U,M_1,M_2), \\
\overline{U} &= \mathsf{U}_{ij}(X,Y,U,M_1,M_2),
\end{align*}
where $\mathsf{Y}_{ij}$ and $\mathsf{U}_{ij}$ satisfy the estimates in Lemma~\ref{lem-rescaling}. This completes the proof of Lemma~\ref{lem-rescaling}.

\section{Bifurcations of Fixed Points in the standard H\'enon Map}
\label{sec-BFPHM}

\subsection{On the Lyapunov coefficient}
\label{sec-LC}

Let $g$ be a $C^4$ diffeomorphism of a closed $C^4$ manifold $M_\mathrm{ph}$, and let $Q$ be a periodic point of period $\mathrm{per}(Q)$. Assume that $Dg^{\mathrm{per}(Q)}(Q)$ has simple unit multipliers
\[
\nu = \cos\psi + \mathrm{i}\sin\psi,
\qquad
\overline\nu = \cos\psi - \mathrm{i}\sin\psi,
\qquad
\psi \in (0,\pi),
\]
and that no other multiplier lies on the unit circle.
Set
\[
\Psi_\mathrm{reg}
:=
\left\{
\psi \in (0,\pi)
\middle|\,
\psi \notin \frac{2\pi}{j}\mathbb{Z}
\text{ for every }
j \in \{1,2,3,4\}
\right\}
:=
(0,\pi) \setminus \left\{\frac{\pi}{2}, \frac{2\pi}{3}\right\}.
\]
Assume in addition that $\psi \in \Psi_\mathrm{reg}$. By the center manifold theorem, $Q$ has a two-dimensional local center manifold $W^c_\mathrm{loc}(Q)$; see \cite{Kelley1967} and \cite[Section~5A]{HPS1977}. This manifold is not unique and in general only finitely smooth; see \cite[Section~5.10.2]{Robinson1999}. Hence we fix one of class $C^4$. After restricting to a smaller neighborhood $W^c \subset W^c_\mathrm{loc}(Q)$ and choosing a $C^4$ complex coordinate $z$ on $W^c$, write
\[
g^{\mathrm{per}(Q)}|_{W^c} : z \longmapsto \widetilde{z}
=
\nu z
+
\sum_{2 \le p+q \le 3} \widetilde{z}^{(pq)} z^p \overline z^q
+
O(|z|^4),
\qquad
p, q \in \mathbb{Z}_{\ge 0}.
\]

\begin{definition}
\label{def-LC}
The quadratic change
\[
w
:=
z
+
\sum_{p+q=2}
\frac{\widetilde{z}^{(pq)}}{\nu - \nu^p \overline\nu^q}
z^p \overline z^q
\]
is well defined, because $\nu - \nu^p \overline\nu^q \neq 0$ for every $(p,q)$ with $p+q=2$. It removes the quadratic terms, so in the $w$ coordinate the map becomes
\[
\widetilde{w}
:=
\nu w
+
\sum_{p+q=3}
\widetilde{w}^{(pq)} w^p \overline w^q
+
O(|w|^4).
\]
See, for example, \cite[Sections~7, 8]{RT1971}, \cite[Chapter~III]{Iooss1979}, \cite[Sections~6, 6A]{MarsdenMcCracken1976}, or \cite[Section~2.8]{Devaney2003}.
Define
\[
\alpha(Q;z) := \widetilde{w}^{(21)},
\qquad
\mathrm{LC}(Q;z) := \Re(\overline\nu \alpha(Q;z)).
\]
We call $\mathrm{LC}(Q;z)$ the \emph{Lyapunov coefficient} of $Q$ in the coordinate $z$.
\end{definition}

\begin{remark}\label{prop-LC-sign}
Because $\psi \in \Psi_\mathrm{reg}$, the cubic change
\[
\zeta
:=
w
+
\sum_{\substack{p+q=3\\(p,q)\neq (2,1)}}
\frac{\widetilde{w}^{(pq)}}{\nu - \nu^p \overline\nu^q}
w^p \overline w^q.
\]
is well defined, because $\nu - \nu^p \overline\nu^q \neq 0$ for every $(p,q)$ with $p+q=3$ and $(p,q)\neq(2,1)$. It removes the nonresonant cubic terms, so
\[
\widetilde{\zeta}
:=
\nu \zeta + \alpha(Q;z)\zeta^2 \overline\zeta + O(|\zeta|^4).
\]
See the references above. Since $|\nu| = 1$,
\[
|\widetilde{\zeta}|^2
=
|\zeta|^2 + 2\mathrm{LC}(Q;z)|\zeta|^4 + O(|\zeta|^5),
\]
hence
\[
|\widetilde{\zeta}|
=
|\zeta| + \mathrm{LC}(Q;z)|\zeta|^3 + O(|\zeta|^4).
\]
Consequently, if $\mathrm{LC}(Q;z) > 0$, then $Q$ is weakly repelling on $W^c_\mathrm{loc}(Q)$, and if $\mathrm{LC}(Q;z) < 0$, then $Q$ is weakly attracting there.
\end{remark}

\begin{proposition}
\label{prop-LC-coordinate-free}
Assume $\psi \in \Psi_\mathrm{reg}$. If $z_1$ and $z_2$ are two local complex coordinates on $W^c$ such that
\[
g^{\mathrm{per}(Q)}|_{W^c} : z_j \longmapsto \widetilde{z}_j = \nu z_j + O(|z_j|^2),
\qquad
j \in \{1,2\},
\]
then there exists $\kappa > 0$ such that
\[
\mathrm{LC}(Q;z_2) = \kappa \, \mathrm{LC}(Q;z_1).
\]
In particular, the sign and the vanishing of $\mathrm{LC}(Q;z)$ are independent of the chosen adapted local complex coordinate.
\end{proposition}

\begin{proof}
For $j \in \{1,2\}$, let $w_j$ be the quadratic normalization from Definition~\ref{def-LC}, and define $\zeta_j$ from $w_j$ exactly as in Remark~\ref{prop-LC-sign}. Then
\[
\widetilde{\zeta_j}
:=
\nu \zeta_j + \alpha_j \zeta_j^2 \overline\zeta_j + O(|\zeta_j|^4),
\qquad
\alpha_j = \alpha(Q;z_j).
\]
Definition~\ref{def-LC} gives
\[
\mathrm{LC}(Q;z_j) = \Re(\overline\nu \alpha_j).
\]
Write
\[
\zeta_1 = \phi(\zeta_2)
=
A \zeta_2 + B \overline\zeta_2 + O(|\zeta_2|^2).
\]
The linear terms in $\phi(\widetilde{\zeta_2}) = \widetilde{\zeta_1}$ give
\[
A \nu \zeta_2 + B \overline\nu \overline\zeta_2
=
\nu A \zeta_2 + \nu B \overline\zeta_2.
\]
Since $\nu \neq \overline\nu$, we get $B = 0$. Because $\phi$ is a local diffeomorphism, this implies $A \neq 0$. Therefore
\[
\zeta_1
:=
A \zeta_2
+
\sum_{p+q=2} B_{pq} \zeta_2^p \overline\zeta_2^q
+
\sum_{p+q=3} C_{pq} \zeta_2^p \overline\zeta_2^q
+
O(|\zeta_2|^4),
\qquad
A \neq 0.
\]
Comparing quadratic terms and nonresonant cubic terms in $\phi(\widetilde{\zeta_2}) = \widetilde{\zeta_1}$, and using again that $\nu - \nu^p \overline\nu^q \neq 0$ for every $p+q=2$ and every $(p,q)\neq(2,1)$ with $p+q=3$, we get
\[
\zeta_1 = A \zeta_2 + C_{21}\zeta_2^2 \overline\zeta_2 + O(|\zeta_2|^4).
\]
Comparing the coefficient of $\zeta_2^2 \overline\zeta_2$ in the two normal forms gives $\alpha_2 = |A|^2 \alpha_1$. Therefore
\[
\mathrm{LC}(Q;z_2) = |A|^2 \mathrm{LC}(Q;z_1).
\]
The claim follows with $\kappa := |A|^2 > 0$.
\end{proof}

\cite[Proposition~4.1]{Tomizawa2025} gives the following formulas. Note that the Lyapunov coefficient in \cite{Tomizawa2025} has the opposite sign to Definition~\ref{def-LC}.
\begin{proposition}\label{prop-LCfmla}
Assume $\psi \in \Psi_\mathrm{reg}$. Then the coefficient $\alpha$ in Definition~\ref{def-LC} is
\begin{equation}
\label{eq-alpha}
\alpha
=
\widetilde{z}^{(21)}
+
|\widetilde{z}^{(02)}|^2
\frac{4\nu - 2\overline\nu^2}{-2 + \nu^3 + \overline\nu^3}
+
|\widetilde{z}^{(11)}|^2
\frac{2 - \overline\nu}{(-1 + \overline\nu)^2}
-
\widetilde{z}^{(11)} \widetilde{z}^{(20)}
\frac{-6 + 2\nu + \overline\nu}{(-1 + \nu)^2}.
\end{equation}
Consequently,
\begin{equation}
\label{eq-nubaralpha}
\overline\nu \alpha
=
\widetilde{z}^{(21)} \overline\nu
+
|\widetilde{z}^{(02)}|^2
\frac{4 - 2\overline\nu^3}{-2 + \nu^3 + \overline\nu^3}
+
|\widetilde{z}^{(11)}|^2
\frac{2\overline\nu - \overline\nu^2}{(-1 + \overline\nu)^2}
-
\widetilde{z}^{(11)} \widetilde{z}^{(20)}
\frac{2 - 6\overline\nu + \overline\nu^2}{(-1 + \nu)^2},
\end{equation}
and thus
\[
\mathrm{LC}(Q;z)
=
\Re\left(
\widetilde{z}^{(21)} \overline\nu
+
|\widetilde{z}^{(02)}|^2
\frac{4 - 2\overline\nu^3}{-2 + \nu^3 + \overline\nu^3}
+
|\widetilde{z}^{(11)}|^2
\frac{2\overline\nu - \overline\nu^2}{(-1 + \overline\nu)^2}
-
\widetilde{z}^{(11)} \widetilde{z}^{(20)}
\frac{2 - 6\overline\nu + \overline\nu^2}{(-1 + \nu)^2}
\right).
\]
\end{proposition}

\subsection{Lyapunov coefficient for the standard H\'enon map}
\label{sec-LCHM}

\begin{lemma}\label{lem-Henon-elliptic}
For the standard H\'enon map $F_{(M_1,M_2)}: (X, Y) \longmapsto (\overline{X}, \overline{Y})$,
\begin{equation}
\label{eq-Hen}
\overline{X} = Y,
\qquad
\overline{Y} = M_1 - M_2 X - Y^2,
\end{equation}
the fixed points are
\[
P_\pm = (X_\pm,X_\pm),
\qquad
X_\pm = \frac{-(1+M_2) \pm \sqrt{(1+M_2)^2 + 4M_1}}{2},
\]
whenever $(1+M_2)^2 + 4M_1 > 0$. The point $P_+$ has multipliers $\mathrm{e}^{\pm \mathrm{i}\psi}$ for a unique $\psi \in (0,\pi)$ if and only if $M_2 = 1$ and $-1 < M_1 < 3$. In that case,
\[
X_+ = -\cos\psi,
\qquad
M_1 = \cos^2\psi - 2\cos\psi.
\]
Moreover, $\psi \longmapsto \cos^2\psi - 2\cos\psi$ is an orientation-preserving diffeomorphism from $(0,\pi)$ onto $(-1,3)$.
\end{lemma}

\begin{proof}
The fixed points satisfy $X = Y$ and $X^2 + (1+M_2)X - M_1 = 0$. At a fixed point $P = (x,x)$,
\[
DF_{(M_1,M_2)}(P)
=
\begin{pmatrix}
0 & 1 \\
-M_2 & -2x
\end{pmatrix},
\qquad
\det(\lambda I - DF_{(M_1,M_2)}(P))
=
\lambda^2 + 2x\lambda + M_2.
\]
If the multipliers of $P_+$ are $\mathrm{e}^{\pm \mathrm{i}\psi}$, then their product gives $M_2 = 1$ and their sum gives $x = -\cos\psi$. Substituting this into $x^2 + 2x - M_1 = 0$ yields $M_1 = \cos^2\psi - 2\cos\psi$. Conversely, if $M_2 = 1$ and $-1 < M_1 < 3$, then $X_+ \in (-1,1)$, so $X_+ = -\cos\psi$ for a unique $\psi \in (0,\pi)$, and the characteristic polynomial is $\lambda^2 - 2(\cos\psi)\lambda + 1$. Finally,
\[
\frac{d}{d\psi}(\cos^2\psi - 2\cos\psi)
=
2\sin\psi(1 - \cos\psi) > 0
\]
for $\psi \in (0,\pi)$.
\end{proof}

\begin{proposition}\label{prop-LC4Hen}
For every $\psi \in \Psi_\mathrm{reg}$, the fixed point $P_+$ of the standard H\'enon map has Lyapunov coefficient
\begin{equation}
\label{eq-LC4Hen}
\mathrm{LC}(P_+;z) = \mathcal{L}(\psi),
\qquad
\mathcal{L}(\psi)
:=
\frac{\cos\psi}{4(-1 + \cos\psi)^2(1 + 2\cos\psi)^2}
\end{equation}
for some local complex coordinate $z$.
\end{proposition}

\begin{proof}
Write $c := \cos\psi$, $s := \sin\psi > 0$, and $\nu := \mathrm{e}^{\mathrm{i}\psi}$. By Lemma~\ref{lem-Henon-elliptic}, at the elliptic parameter we have $M_2 = 1$, $M_1 = c^2 - 2c$, and $P_+ = (-c,-c)$. After the translation $X \longmapsto X + c$ and $Y \longmapsto Y + c$, the map becomes
\[
\overline{X} = Y,
\qquad
\overline{Y} = -X + 2cY - Y^2.
\]
Set
\[
X = U,
\qquad
Y = cU - sV.
\]
Then
\[
\overline{U} = cU - sV,
\qquad
\overline{V} = sU + cV + \frac{c^2}{s}U^2 - 2cUV + sV^2.
\]
In the complex coordinate $z = U + \mathrm{i}V$, this is
\[
\widetilde{z}
=
\nu z
+
\frac{\mathrm{i}\nu^2}{4s} z^2
+
\frac{\mathrm{i}}{2s} z\overline z
+
\frac{\mathrm{i}\overline\nu^2}{4s} \overline z^2.
\]
Hence
\[
\widetilde{z}^{(20)} = \frac{\mathrm{i}\nu^2}{4s},
\qquad
\widetilde{z}^{(11)} = \frac{\mathrm{i}}{2s},
\qquad
\widetilde{z}^{(02)} = \frac{\mathrm{i}\overline\nu^2}{4s},
\qquad
\widetilde{z}^{(21)} = 0.
\]
Using $\nu + \overline\nu = 2c$, $\nu\overline\nu = 1$, and $s^2 = 1 - c^2$, one can see that
\begin{align*}
\Re \left( |\widetilde{z}^{(02)}|^2 \frac{4 - 2\overline{\nu}^3}{-2 + \nu^3 + \overline{\nu}^3} \right) &= \frac{-2 - 3c + 4c^3}{16(1 + c)(-1 + c)^2(1 + 2c)^2}, \\
\Re \left( |\widetilde{z}^{(11)}|^2 \frac{2 \overline{\nu} - \overline{\nu}^2}{(-1 + \overline{\nu})^2} \right) &= \frac{-2 + c}{8(1 + c)(-1 + c)^2}, \\
\Re \left( - \widetilde{z}^{(11)} \widetilde{z}^{(20)}\frac{2 - 6 \overline{\nu} + \overline{\nu}^2}{(-1 + \nu)^2} \right) &= \frac{-3(-2 + c)}{16(1 + c)(-1 + c)^2}.
\end{align*}
Thus, Proposition~\ref{prop-LCfmla} gives \eqref{eq-LC4Hen}.
\end{proof}
\subsection{The bifurcation diagram for the standard H\'enon map}

The following standard facts are known for the standard H\'enon map~\eqref{eq-Hen}; see, for example, \cite{HitzlZele1985}. See also Figure~\ref{fig-bifdiam}.
\begin{itemize}
\item The curve $L^+ := \left\{M_1 = -\frac{(1+M_2)^2}{4}\right\}$ corresponds to the moment of a saddle-node bifurcation.
\item Whenever $(1+M_2)^2 + 4M_1 > 0$, the map has exactly two fixed points $P_\pm = (X_\pm, X_\pm)$, as in Lemma~\ref{lem-Henon-elliptic} with $X_- < X_+$.
\item The curve $L^- := \left\{M_1 = \frac{3(1+M_2)^2}{4}\right\}$ corresponds to the moment of a period-doubling bifurcation at $P_+$.
\item The curve $L^\omega := \{M_2 = 1, -1 < M_1 < 3\}$ corresponds to the moment of a Neimark-Sacker (Andronov-Hopf) bifurcation at $P_+$.
\item The points $B^{++} := (-1,1)$, $B^{--} := (3,1)$, and $B^{+-} := (0,-1)$ correspond to the multiplier pairs $(1,1)$, $(-1,-1)$, and $(1,-1)$, respectively, at $P_+$.
\end{itemize}
Because $M_1 = \cos^2\psi - 2\cos\psi$ in Lemma~\ref{lem-Henon-elliptic}, the points $B^{++}$, $C_1^\omega = (0,1)$, $C_2^\omega = \left(\frac{5}{4},1\right)$, and $B^{--}$ correspond to $\psi = 0$, $\psi = \pi/2$, $\psi = 2\pi/3$, and $\psi = \pi$, respectively. Since the function $\mathcal{L}(\psi)$ in \eqref{eq-LC4Hen} is shown in Figure~\ref{fig-bifdiam}-(a), on the curve $L^\omega$, the fixed point $P_+$ is weakly repelling for $-1 < M_1 < 0$ and weakly attracting for $0 < M_1 < 3$ with $M_1 \neq 5/4$.

The following proposition follows from Proposition~\ref{prop-LC4Hen} and the Neimark-Sacker (Andronov-Hopf) theorem, since $\det DF_{(M_1,M_2)} \equiv M_2$; see \cite[theorem~7.2]{RT1971}, \cite[theorem~6.2]{MarsdenMcCracken1976}, or \cite[Section~4]{K2023}.

\begin{proposition}\label{prop-Henon-NS}
Let $M_1 = L^*(M_2)$ be a $C^\infty$ curve defined for $|M_2 - 1| < \varepsilon_*$ whose graph is transverse to $L^\omega$ at $(M_1^*,1) \in L^\omega \setminus \{C_1^\omega,C_2^\omega\}$. Then, the hyperbolic continuation of $P_+$ undergoes a Neimark-Sacker (Andronov-Hopf) bifurcation at $M_2 = 1$:
\begin{itemize}
\item If $-1 < M_1^* < 0$, then for every $M_2 < 1$ sufficiently close to $1$, the map $F_{(L^*(M_2), M_2)}$ has a repelling normally hyperbolic invariant circle;
\item if $0 < M_1^* < 3$ and $M_1^* \neq 5/4$, then for every $M_2 > 1$ sufficiently close to $1$, it has an attracting normally hyperbolic invariant circle.
\end{itemize}
\end{proposition}

\section{The proof of the main theorem}
\label{sec-PMT}

\subsection{Attracting periodic circles arbitrarily near cycles}

\begin{definition}
We say that $(f,\Gamma^*)$ is \emph{orientable} if it satisfies the following condition \textup{(P4)}:
\begin{itemize}
\item \textup{(P4)} For the signature $\varsigma_{ij}$ defined in Lemma~\ref{lem-rescaling}, there exists $(i, j) \in \mathcal{I}$ such that $\varsigma_{ij} = 1$.
\end{itemize}
\end{definition}

\begin{remark}\leavevmode
\label{rem-orientable}
\begin{enumerate}
\item[(1)] The property \textup{(P4)} implies infinitely many $(i, j) \in \mathcal{I}$ such that $\varsigma_{ij} = 1$. Indeed, $\varsigma_{ij}$ depends only on $i,j \bmod 2$, and $\mathcal{I}$ contains infinitely many pairs with any prescribed such remainders.
\item[(2)] The property \textup{(P4)} does not depend on the choice of approximately linearized coordinate systems or base points. This is clear from $-B_{21}C_{21}J_{12} = \det \mathbf{J}_{12}\det \mathbf{J}_{21}$.
\end{enumerate}
\end{remark}

\begin{theorem}
\label{thm-circ}
Let $r \in \mathbb{Z}_{\ge 5} \sqcup \{\infty\}$ and $\dim(M_\mathrm{ph}) \ge 3$. Suppose that $f \in \mathrm{Diff}^r(M_\mathrm{ph})$ has a centrally dissipative-expanding transversal and non-transversal heteroclinic cycle $\Gamma^*$ of type two bi-saddles with one-dimensional unstable directions, and that $(f,\Gamma^*)$ satisfies \textup{(P1)}--\textup{(P4)}. Then, for every sequence $\{s_k\}_{k=1}^\infty \subset \mathbb{Z}_{\ge 5}$ with
\[
s_k \xrightarrow{k\to\infty} \infty \text{ if } r = \infty,
\qquad
s_k = r \text{ if } r < \infty,
\]
there exist sequences
\[
\{\varepsilon_k\}_{k=1}^\infty \subset R_\mathrm{prm}^*,
\qquad
\{\tau_k\}_{k=1}^\infty \subset \mathbb{Z}_{>0},
\]
such that
\[
\varepsilon_k \xrightarrow{k\to\infty} \varepsilon^*,
\qquad
\tau_k \xrightarrow{k\to\infty} \infty,
\]
and $f_{\varepsilon_k}$ has an attracting normally hyperbolic periodic $C^{s_k}$ circle of period $\tau_k$ for every $k$.
\end{theorem}

\begin{proof}
Fix $\{s_k\}_{k=1}^\infty$. By Remark~\ref{rem-orientable}, we can choose $\{(i_k,j_k)\}_{k=1}^\infty \subset \mathcal{I}$ such that
\[
\varsigma_{i_kj_k} = 1,
\qquad
\min\{i_k,j_k\} \xrightarrow{k\to\infty} \infty.
\]

Choose $M_1^\dagger \in (0,3) \setminus \{5/4\}$. For all large $k$, let $\widehat{L}_k^\omega \subset \Delta_{i_kj_k}$ be the bifurcation curve corresponding to $L^\omega$, let $\varepsilon_k^\dagger$ be the point corresponding to $(M_1^\dagger,1)$ under the rescaling of Lemma~\ref{lem-rescaling}, and let $Q_k^\dagger$ be a fixed point of $T_{i_kj_k}(\varepsilon_k^\dagger;\mathfrak{F}_\mathrm{obj})$ whose multipliers are $\nu_k,\overline{\nu_k}$ and whose other multipliers have modulus $< 1$.

By the center manifold theorem, the local dynamics of $T_{i_kj_k}(\varepsilon_k^\dagger;\mathfrak{F}_\mathrm{obj})$ at $Q_k^\dagger$ reduces to a $C^{s_k}$ surface map. By Lemma~\ref{lem-rescaling} and the estimate at the end of its proof, this center map is $C^4$-close to the standard H\'enon map at $(M_1^\dagger,1)$ for all large $k$. Hence, by Definition~\ref{def-LC}, Proposition~\ref{prop-LC4Hen}, and Proposition~\ref{prop-LC-coordinate-free}, its Lyapunov coefficient is negative for all large $k$.

Let $\widehat{\Gamma}_k \subset \Delta_{i_kj_k}$ be a $C^\infty$ curve through $\varepsilon_k^\dagger$ transverse to $\widehat{L}_k^\omega$. Apply the Neimark-Sacker (Andronov-Hopf) theorem to the $C^r$ one-parameter family $T_{i_kj_k}|_{\widehat{\Gamma}_k}$ at $\varepsilon_k^\dagger$; see \cite[theorem~7.2]{RT1971}, \cite[theorem~6.2]{MarsdenMcCracken1976}, or \cite[Section~4]{K2023}. Note that this theorem requires $r \ge 5$. Then there exists a parameter
\[
\varepsilon_k \in \widehat{\Gamma}_k
\]
arbitrarily close to $\varepsilon_k^\dagger$ such that $T_{i_kj_k}(\varepsilon_k;\mathfrak{F}_\mathrm{obj})$ has an attracting invariant $C^{s_k}$ circle $C_k$. Since the remaining multipliers at $Q_k^\dagger$ are strictly inside the unit disk, $C_k$ is normally hyperbolic. By Lemma~\ref{lem-rescaling},
\[
\sup_{\varepsilon \in \Delta_{i_kj_k}} \|\varepsilon - \varepsilon^*\| \xrightarrow{k\to\infty} 0,
\]
so $\varepsilon_k \xrightarrow{k\to\infty} \varepsilon^*$.

Set
\[
\tau_k := i_k\mathrm{per}(O_1^*) + N_{1 \to 2} + j_k\mathrm{per}(O_2^*) + N_{2 \to 1}.
\]
Since
\[
T_{i_kj_k}(\varepsilon_k;\mathfrak{F}_\mathrm{obj})
=
f_{\varepsilon_k}^{\tau_k}\big|_{\widetilde{\Lambda}_1^\mathrm{in}}
\]
is the first-return map along the fixed itinerary, we have
\[
f_{\varepsilon_k}^m(C_k) \cap C_k = \emptyset
\qquad
(1 \le m \le \tau_k - 1).
\]
Hence $C_k$ is an attracting normally hyperbolic periodic $C^{s_k}$ circle for $f_{\varepsilon_k}$, with period $\tau_k$. Since $\min\{i_k,j_k\} \xrightarrow{k\to\infty} \infty$, we have $\tau_k \xrightarrow{k\to\infty} \infty$.
\end{proof}

The invariance of normally hyperbolic invariant manifolds implies the following consequence; see \cite{HPS1977}.

\begin{corollary}
\label{cor-circ}
Under the assumptions of Theorem~\ref{thm-circ}, fix a sequence $\{s_k\}_{k=1}^\infty \subset \mathbb{Z}_{\ge 5}$ such that
\[
s_k \xrightarrow{k\to\infty} \infty \text{ if } r = \infty,
\qquad
s_k = r \text{ if } r < \infty.
\]
Then there exist sequences
\[
\{f_k\}_{k=1}^\infty \subset \mathrm{Diff}^r(M_\mathrm{ph}),
\qquad
\{\mathcal{U}_k\}_{k=1}^\infty,
\qquad
\{\tau_k\}_{k=1}^\infty \subset \mathbb{Z}_{>0},
\]
such that
\[
f_k \xrightarrow{k\to\infty} f \quad \text{in } C^r,
\qquad
\tau_k \xrightarrow{k\to\infty} \infty,
\]
each $\mathcal{U}_k$ is a $C^r$ neighborhood of $f_k$, and every $g \in \mathcal{U}_k$ has an attracting normally hyperbolic periodic $C^{s_k}$ circle of period $\tau_k$.
\end{corollary}

By Proposition~\ref{prop-generic}, \textup{(P1)}--\textup{(P3)} are achieved by arbitrarily small $C^r$ perturbations. For \textup{(P4)}, it suffices to replace the global maps as in \cite[Section~4]{GSS2002}. Thus Corollary~\ref{cor-circ} applies after an arbitrarily small perturbation.

\begin{corollary}
\label{cor-main-from-circ}
Let $r \in \mathbb{Z}_{\ge 5} \sqcup \{\infty\}$, and suppose that $f \in \mathrm{Diff}^r(M_\mathrm{ph})$ has a centrally dissipative-expanding transversal and non-transversal heteroclinic cycle of type two bi-saddles with one-dimensional unstable directions. Then, for every sequence $\{s_k\}_{k=1}^\infty \subset \mathbb{Z}_{\ge 5}$ with
\[
s_k \xrightarrow{k\to\infty} \infty \text{ if } r = \infty,
\qquad
s_k = r \text{ if } r < \infty,
\]
there exist sequences
\[
\{f_k\}_{k=1}^\infty \subset \mathrm{Diff}^r(M_\mathrm{ph}),
\qquad
\{\mathcal{U}_k\}_{k=1}^\infty,
\qquad
\{\tau_k\}_{k=1}^\infty \subset \mathbb{Z}_{>0},
\]
such that
\[
f_k \xrightarrow{k\to\infty} f \quad \text{in } C^r,
\qquad
\tau_k \xrightarrow{k\to\infty} \infty,
\]
each $\mathcal{U}_k$ is a $C^r$ neighborhood of $f_k$, and every $g \in \mathcal{U}_k$ has an attracting normally hyperbolic periodic $C^{s_k}$ circle of period $\tau_k$.
\end{corollary}

\subsection{Proof of the main theorem}

\begin{proof}[Proof of Theorem~\ref{thm-main}]
If $r < \infty$, set $s_n := r$ for all $n \in \mathbb{Z}_{>0}$. If $r = \infty$, fix a sequence $\{s_n\}_{n=1}^\infty \subset \mathbb{Z}_{\ge 5}$ as in the statement. Replacing $s_n$ with $\max\{s_1,\dots,s_n\}$, we may assume that $\{s_n\}_{n=1}^\infty$ is nondecreasing.

Applying the Newhouse domain construction to our cycle, as in \cite{GTS1997} and its multidimensional version \cite[Section~1.2]{GST2008}, we obtain an open set $\mathcal{N} \subset \mathrm{Diff}^r(M_\mathrm{ph})$ and a dense subset $\mathcal{D}_\mathrm{cycle} \subset \mathcal{N}$ such that $f \in \overline{\mathcal{N}}$ and every $g \in \mathcal{D}_\mathrm{cycle}$ has a centrally dissipative-expanding transversal and non-transversal heteroclinic cycle of type two bi-saddles with one-dimensional unstable directions.

For each $g \in \mathcal{D}_\mathrm{cycle}$, Corollary~\ref{cor-main-from-circ} gives sequences $\{f_{g,k}\}_{k=1}^\infty$, $\{\mathcal{U}_{g,k}\}_{k=1}^\infty$, and $\{\tau_{g,k}\}_{k=1}^\infty$ such that $f_{g,k} \xrightarrow{k\to\infty} g$, $\tau_{g,k} \xrightarrow{k\to\infty} \infty$, and every $h \in \mathcal{U}_{g,k}$ has an attracting normally hyperbolic periodic $C^{s_k}$ circle of period $\tau_{g,k}$.

For each $n \in \mathbb{Z}_{>0}$, set
\[
\mathcal{O}_n := \bigcup_{g \in \mathcal{D}_\mathrm{cycle}} \, \bigcup_{\substack{k \in \mathbb{Z}_{>0} \\ k \ge n,\ \tau_{g,k} \ge n}} \mathcal{U}_{g,k}.
\]
Then $\mathcal{O}_n$ is open. It is dense in $\mathcal{N}$: if $\mathcal{W} \subset \mathcal{N}$ is nonempty and open, choose $g \in \mathcal{D}_\mathrm{cycle} \cap \mathcal{W}$. Since $f_{g,k} \xrightarrow{k\to\infty} g$ and $\tau_{g,k} \xrightarrow{k\to\infty} \infty$, for all large $k$ we have $k \ge n$, $f_{g,k} \in \mathcal{W}$, and $\tau_{g,k} \ge n$. Hence $\mathcal{W} \cap \mathcal{U}_{g,k} \subset \mathcal{W} \cap \mathcal{O}_n$ is nonempty.
\[
\mathcal{R} := \bigcap_{n \in \mathbb{Z}_{>0}} \mathcal{O}_n
\]
is residual in $\mathcal{N}$. If $h \in \mathcal{R}$, then for every $n \in \mathbb{Z}_{>0}$ there exist $g \in \mathcal{D}_\mathrm{cycle}$ and $k \ge n$ such that $h \in \mathcal{U}_{g,k}$. Thus $h$ has an attracting normally hyperbolic periodic $C^{s_k}$ circle of period at least $n$. Since $\{s_n\}_{n=1}^\infty$ is nondecreasing and $k \ge n$, this circle is also $C^{s_n}$. Taking $\mathcal{U} := \mathcal{N}$ proves Theorem~\ref{thm-main}.
\end{proof}

\section*{Acknowledgments}
I am grateful to Shuhei Hayashi and Sogo Murakami for valuable discussions. I also thank Shin Kiriki, Yushi Nakano, and Teruhiko Soma. Without their involvement, this work would not have begun.
%%%%%%%%%%%%%%%%%%%%%%%%%%% body end %%%%%%%%%%%%%%%%%%%%%%%%%%%%%%%%

\bibliography{references}

@article{G2002,
  author = {Gonchenko, V. S.},
  title = {On Bifurcations of Two-Dimensional Diffeomorphisms with a Homoclinic Tangency of Manifolds of a ``Neutral'' Saddle},
  journal = {Proceedings of the Steklov Institute of Mathematics},
  volume = {236},
  pages = {86--93},
  year = {2002}
}

@techreport{GG2000,
  author = {Gonchenko, S. V. and Gonchenko, V. S.},
  title = {On {Andronov--Hopf} Bifurcations of Two-Dimensional Diffeomorphisms with Homoclinic Tangencies},
  institution = {Weierstrass Institute for Applied Analysis and Stochastics},
  address = {Berlin},
  number = {556},
  year = {2000}
}

@article{GG2004,
  author = {Gonchenko, S. V. and Gonchenko, V. S.},
  title = {On Bifurcations of Birth of Closed Invariant Curves in the Case of Two-Dimensional Diffeomorphisms with Homoclinic Tangencies},
  journal = {Proceedings of the Steklov Institute of Mathematics},
  volume = {244},
  pages = {80--105},
  year = {2004}
}

@article{GKM2005,
  author = {Gonchenko, V. S. and Kuznetsov, Yu. A. and Meijer, H. G. E.},
  title = {Generalized {H\'enon} Map and Bifurcations of Homoclinic Tangencies},
  journal = {SIAM Journal on Applied Dynamical Systems},
  volume = {4},
  number = {2},
  pages = {407--436},
  year = {2005},
  doi = {10.1137/04060487X}
}

@article{GO2005,
  author = {Gonchenko, V. S. and Ovsyannikov, I. I.},
  title = {On Bifurcations of Three-Dimensional Diffeomorphisms with a Homoclinic Tangency to a ``Neutral'' Saddle Fixed Point},
  journal = {Journal of Mathematical Sciences},
  volume = {128},
  number = {2},
  pages = {2774--2777},
  year = {2005},
  doi = {10.1007/s10958-005-0229-5}
}

@inproceedings{GSS2002,
  author = {Gonchenko, S. V. and Shilnikov, L. P. and Sten'kin, O. V.},
  title = {On {Newhouse} Regions with Infinitely Many Stable and Unstable Invariant Tori},
  booktitle = {Progress in Nonlinear Science. Vol. 1: Mathematical Problems of Nonlinear Dynamics},
  note = {Proceedings of the International Conference, Nizhni Novgorod, 2001},
  pages = {80--102},
  address = {Nizhni Novgorod},
  year = {2002}
}

@article{GSS2006,
  author = {Gonchenko, S. V. and Sten'kin, O. V. and Shilnikov, L. P.},
  title = {On the Existence of Infinitely Many Stable and Unstable Invariant Tori for Systems from {Newhouse} Regions with Heteroclinic Tangencies},
  journal = {Russian Journal of Nonlinear Dynamics},
  volume = {2},
  number = {1},
  pages = {3--25},
  year = {2006},
  doi = {10.20537/nd0601001}
}

@article{GTS1997,
  author = {Gonchenko, S. V. and Turaev, D. V. and Shilnikov, L. P.},
  title = {On {Newhouse} Domains of Two-Dimensional Diffeomorphisms That Are Close to a Diffeomorphism with a Structurally Unstable Heteroclinic Contour},
  journal = {Proceedings of the Steklov Institute of Mathematics},
  volume = {216},
  pages = {70--118},
  year = {1997}
}

@article{GST2008,
  author = {Gonchenko, S. V. and Shilnikov, L. P. and Turaev, D. V.},
  title = {On Dynamical Properties of Multidimensional Diffeomorphisms from {Newhouse} Regions: {I}},
  journal = {Nonlinearity},
  volume = {21},
  number = {5},
  pages = {923--972},
  year = {2008},
  doi = {10.1088/0951-7715/21/5/003}
}

@book{HPS1977,
  author = {Hirsch, Morris W. and Pugh, Charles C. and Shub, Michael},
  title = {Invariant Manifolds},
  series = {Lecture Notes in Mathematics},
  volume = {583},
  publisher = {Springer},
  address = {Berlin},
  year = {1977},
  doi = {10.1007/BFb0092042}
}

@incollection{Kelley1967,
  author = {Kelley, Al},
  title = {The stable, center-stable, center, center-unstable, and unstable manifolds},
  booktitle = {Transversal Mappings and Flows},
  editor = {Abraham, Ralph and Robbin, Joel},
  publisher = {W. A. Benjamin},
  address = {New York},
  year = {1967},
  pages = {134--154},
  note = {Appendix C}
}

@article{HitzlZele1985,
  author = {Hitzl, D. L. and Zele, F.},
  title = {An exploration of the {H\'enon} quadratic map},
  journal = {Physica D: Nonlinear Phenomena},
  volume = {14},
  number = {3},
  pages = {305--326},
  year = {1985},
  doi = {10.1016/0167-2789(85)90092-2}
}

@book{Iooss1979,
  author = {Iooss, G{\'e}rard},
  title = {Bifurcation of Maps and Applications},
  series = {North-Holland Mathematics Studies},
  volume = {36},
  publisher = {North-Holland},
  address = {Amsterdam},
  year = {1979}
}

@article{Newhouse1974,
  author = {{Newhouse}, Sheldon E.},
  title = {Diffeomorphisms with Infinitely Many Sinks},
  journal = {Topology},
  volume = {13},
  number = {1},
  pages = {9--18},
  year = {1974},
  doi = {10.1016/0040-9383(74)90035-2}
}

@article{PalisViana1994,
  author = {Palis, Jacob and Viana, Marcelo},
  title = {High Dimension Diffeomorphisms Displaying Infinitely Many Periodic Attractors},
  journal = {Annals of Mathematics},
  volume = {140},
  number = {1},
  pages = {207--250},
  year = {1994},
  doi = {10.2307/2118546}
}

@book{Robinson1999,
  author = {Robinson, R. Clark},
  title = {Dynamical Systems: Stability, Symbolic Dynamics, and Chaos},
  edition = {2},
  publisher = {CRC Press},
  address = {Boca Raton, FL},
  year = {1999},
  isbn = {978-0-8493-8495-0}
}

@article{Romero1995,
  author = {Romero, Nestor},
  title = {Persistence of Homoclinic Tangencies in Higher Dimensions},
  journal = {Ergodic Theory and Dynamical Systems},
  volume = {15},
  number = {4},
  pages = {735--757},
  year = {1995},
  doi = {10.1017/S0143385700008634}
}

@article{Tatjer2001,
  author = {Tatjer, Joan Carles},
  title = {Three-Dimensional Dissipative Diffeomorphisms with Homoclinic Tangencies},
  journal = {Ergodic Theory and Dynamical Systems},
  volume = {21},
  number = {1},
  pages = {249--302},
  year = {2001},
  doi = {10.1017/S0143385701001146}
}

@article{GOT2012,
  author = {Gonchenko, S. V. and Ovsyannikov, I. I. and Turaev, D.},
  title = {On the Effect of Invisibility of Stable Periodic Orbits at Homoclinic Bifurcations},
  journal = {Physica D: Nonlinear Phenomena},
  volume = {241},
  number = {13},
  pages = {1115--1122},
  year = {2012},
  doi = {10.1016/j.physd.2012.04.007}
}

@article{MoraRuiz2011,
  author = {Mora, Leonardo and Ruiz, Bladismir},
  title = {Diffeomorphisms with Infinitely Many Irrational Invariant Curves},
  journal = {Ergodic Theory and Dynamical Systems},
  volume = {31},
  number = {5},
  pages = {1517--1535},
  year = {2011},
  doi = {10.1017/S0143385710000594}
}

@book{SSTC2001,
  author = {Shilnikov, Leonid P. and Shilnikov, Andrey L. and Turaev, Dmitry V. and Chua, Leon O.},
  title = {Methods of Qualitative Theory in Nonlinear Dynamics. Part {I}},
  edition = {2},
  publisher = {World Scientific},
  address = {Singapore},
  year = {2001}
}

@article{RT1971,
  author = {Ruelle, David and Takens, Floris},
  title = {On the Nature of Turbulence},
  journal = {Communications in Mathematical Physics},
  volume = {20},
  pages = {167--192},
  year = {1971},
  doi = {10.1007/BF01646553}
}

@book{MarsdenMcCracken1976,
  author = {Marsden, Jerrold E. and McCracken, Marjorie},
  title = {The {Hopf} Bifurcation and Its Applications},
  series = {Applied Mathematical Sciences},
  volume = {19},
  publisher = {Springer-Verlag},
  address = {New York},
  year = {1976},
  doi = {10.1007/978-1-4612-6374-6}
}

@book{Devaney2003,
  author = {Devaney, Robert L.},
  title = {An Introduction to Chaotic Dynamical Systems},
  edition = {2},
  publisher = {Westview Press},
  address = {Boulder, CO},
  year = {2003},
  isbn = {9780813340852}
}

@misc{Tomizawa2025,
  author = {Tomizawa, Shuntaro},
  title = {Heterodimensional cycles derived from homoclinic tangencies via Hopf bifurcations},
  year = {2025},
  eprint = {2505.12596},
  archivePrefix = {arXiv},
  primaryClass = {math.DS},
  doi = {10.48550/arXiv.2505.12596}
}

@book{K2023,
  author = {Kuznetsov, Yuri A.},
  title = {Elements of Applied Bifurcation Theory},
  edition = {4},
  series = {Applied Mathematical Sciences},
  volume = {112},
  publisher = {Springer},
  address = {Cham},
  year = {2023},
  doi = {10.1007/978-3-031-22007-4}
}
\bibliographystyle{alpha}

\end{document}